\documentclass[12pt]{amsart}
\usepackage[all]{xy}
\usepackage[parfill]{parskip}

\usepackage{verbatim}
\usepackage{color}

\usepackage{amsmath, amscd, graphicx, latexsym, hyperref, times}
\usepackage{esint} 
\usepackage{graphicx}
\usepackage[abs]{overpic}
\usepackage{tikz}
\usepackage{tikz-cd}
\usetikzlibrary{arrows, patterns}
\usepackage{soul}

\oddsidemargin .25in \theoremstyle{plain}

\theoremstyle{plain}      
\newtheorem{theorem}{Theorem}[section]
\newtheorem{theoremx}{Theorem}

\newtheorem{proposition}[theorem]{Proposition}
\newtheorem{lemma}[theorem]{Lemma}
\newtheorem{corollary}[theorem]{Corollary}

\theoremstyle{definition}

\theoremstyle{remark}     
\newtheorem{remark}[theorem]{Remark}

\allowdisplaybreaks[3]

\numberwithin{equation}{section}

\newcommand{\R}{\mathbb{R}}

\newcommand{\dd}{\mathop{}\!\mathrm{d}}
\newcommand{\dx}{\mathop{}\!\mathrm{d}x}
\newcommand{\dy}{\mathop{}\!\mathrm{d}y}
\newcommand{\nablag}{\nabla^g} 
\usepackage{cleveref}
\usepackage{etoolbox}
\crefname{theorem}{Theorem}{Theorems}

\crefname{lemma}{Lemma}{Lemmas}
\AtBeginEnvironment{lemma}{\crefalias{theorem}{lemma}}

\crefname{proposition}{Proposition}{Propositions}
\AtBeginEnvironment{proposition}{\crefalias{theorem}{proposition}}

\crefname{corollary}{Corollary}{Corollaries}
\AtBeginEnvironment{corollary}{\crefalias{theorem}{corollary}}  

\crefname{remark}{Remark}{Remarks}
\AtBeginEnvironment{remark}{\crefalias{theorem}{remark}}

\crefname{conjecture}{Conjecture}{Conjectures}
\AtBeginEnvironment{conjecture}{\crefalias{theorem}{conjecture}}

\crefname{example}{Example}{Examples}
\AtBeginEnvironment{example}{\crefalias{theorem}{example}}

\crefname{definition}{Definition}{Definitions}
\AtBeginEnvironment{definition}{\crefalias{theorem}{definition}}

\crefname{equation}{}{}
\crefname{section}{Section}{Section}
\crefname{figure}{Figure}{Figure}
\usepackage{appendix}
\begin{document}

\title[Dimension of polynomial growth harmonic functions]{Dimension of polynomial growth harmonic functions on locally conformally flat manifolds with nonnegative Ricci curvature}
\author{Xiaohan Cai}
\address{School of Mathematical Sciences, Shanghai Jiao Tong University, Shanghai 200240, China}
\email{xiaohancai@sjtu.edu.cn}
\author{Mijia Lai}
\address{School of Mathematical Sciences, Shanghai Jiao Tong University, Shanghai 200240, China}
\email{laimijia@sjtu.edu.cn}

\thanks{}

\begin{abstract}
Let $\mathcal{H}_d(M)$ denote the space of harmonic functions with polynomial growth of degree at most $d$ on a complete Riemannian manifold $(M,g)$. Yau raised two fundamental questions regarding $\mathcal{H}_d(M)$ on complete manifolds with nonnegative Ricci curvature. The first question is the finite dimensionality of $\mathcal{H}_d(M)$, which was confirmed by Colding and Minicozzi. The second question asks whether a sharp upper bound given by its Euclidean analog  $\operatorname{dim}\mathcal{H}_{d}(\mathbb{R}^n)$ holds. We verify that the second question is true on locally conformally flat manifolds. Indeed, one can precisely determine the value of $\dim \mathcal{H}_d(M)$ case by case. 
\end{abstract}
\maketitle
\section{Introduction}

The study of harmonic functions on complete noncompact Riemannian manifolds with nonnegative Ricci curvature is a central topic connecting geometric analysis, potential theory, and the large-scale geometry of manifolds. The subject originated with Yau’s classical Liouville theorem~\cite{Y1}, which establishes that every positive harmonic function must be constant on a complete manifold with nonnegative Ricci curvature.  The local gradient estimates of Cheng and Yau~\cite{CY} subsequently  provided a fundamental analytic tool, producing Harnack inequalities and quantitative control of positive harmonic functions under lower Ricci curvature bounds. Li and Tam \cite{LT0, LT3} developed a complementary potential-theoretic approach that relates spaces of bounded and positive harmonic functions to the number and geometry of the ends of a manifold. 

Yau's seminal result also opened an interesting direction of research : the study of the space of harmonic functions with at most polynomial growth.  
For every fixed $d>0$, let
\[
\mathcal{H}_d(M):=\{u:\Delta u=0,\ |u(x)|\leq C(1+r(x))^d\}
\]
denote the space of harmonic functions of polynomial growth of at most degree $d$. Yau raised two fundamental questions regarding $\mathcal{H}_d(M)$ when $M$ has nonnegative Ricci curvature: 
\begin{enumerate}
    \item[Q1] {\bf finite dimensionality}: is it true that $\dim \mathcal{H}_d(M)<\infty$?
    \item[Q2] 
{\bf sharp Euclidean comparison}: is it true that $\dim \mathcal{H}_d(M^n)\le \dim \mathcal{H}_d(\mathbb R^n)$?
\end{enumerate}

For brevity, write $h_d(M)=\dim \mathcal{H}_d(M)$. It is a classical result that any harmonic function with polynomial growth on $\mathbb{R}^n$ must be a harmonic polynomial. For any fixed integer $d\ge1$, it is a fun exercise to show that 
\[
h_d(\mathbb R^n)=\binom{n+d-1}{n-1}+
\binom{n+d-2}{n-1}.
\]

Regarding {\bf Q1}, Li--Tam \cite{LT2} first established this for surfaces with finite total curvature. Colding and Minicozzi \cite{CM1} verified the finite dimensionality of $\mathcal{H}_d(M)$ in general dimension. Li~\cite{L1} also gave an independent proof. 
Later, Colding and Minicozzi \cite{CM2} proved a Weyl type bound 
\[
h_d(M^n)\le C(n)d^{n-1}+ o(d^{n-1}),
\]
which is asymptotically sharp in the power of $d$ since 
$h_d(\mathbb{R}^n)\sim \frac{2}{(n-1)!} d^{n-1}$ as $d\to \infty$. 

For {\bf Q2}, the case $d=1$ is known by ~\cite{LT1}. One also has rigidity in view of ~\cite{CCM}. More precisely, let $M^n$ be a complete noncompact manifold with nonnegative Ricci curvature, and suppose that $h_1(M^n)=h_1(\mathbb{R}^n)=n+1$, then $M$ must be isometric to $\mathbb{R}^n$. 

However, for $d>1$, the Euclidean comparison in general is false. Donnelly \cite{D} constructed, in every dimension $n\ge5$, a complete noncompact manifold with $\operatorname{Ric}\ge0$ carrying more than $n+1$ linearly independent harmonic functions of subquadratic growth. Hence, for a suitable $1<d<2$, Donnelly’s example shows
\[
h_d(M^n)>h_d(\mathbb R^n)=n+1.
\]
Nevertheless, the integer-degree Euclidean comparison remains an open question. Under the assumption of maximal volume growth and uniqueness of the tangent cone at infinity, Huang~\cite{H1,H2} established the Euclidean comparison asymptotically with a sharp coefficient, i.e., 
\[
\lim_{d\to \infty} \frac{h_d(M^n)}{d^{n-1}}= \frac{2\beta}{(n-1)!}, 
\]
where $\beta$ is the asymptotic volume ratio of $M$.
Li--Wang \cite{LW} proved a similar result when $M$ has nonnegative sectional curvature. Indeed, the Euclidean comparison is still an open question on non-negatively curved manifolds. There is a recent breakthrough by Lin-Wang-Xu \cite{LWX}, in which they showed that on a three-dimensional manifold with nonnegative sectional curvature and $\beta>0$, for any integer $d\geq 0$, $h_d(M^3)\leq h_d(\mathbb{R}^3)$. Moreover, if the equality holds for some $d\geq 1$, then $M^3$ is isometric to $\mathbb{R}^3$. 

In this paper, we give an affirmative answer to {\bf Q2} on locally conformally flat manifolds:  

\begin{theoremx} \label{thm:A}
Let $(M^n,g) (n\geq 3)$ be a complete noncompact locally conformally flat manifolds with nonnegative Ricci curvature, then for each $d\geq 0$,
\[
h_d(M^n) \leq h_d (\mathbb{R}^n). 
\]
Moreover, if the equality holds for some integer $d\geq 1$, then $M$ is isometric to $\mathbb{R}^n$.    
\end{theoremx}

Let  $(M^n,g)$ be a complete noncompact manifold with nonnegative Ricci curvature. Denote by $\beta$ the asymptotic volume ratio of $(M,g)$, i.e., \[
    \beta=\lim_{r\to \infty}\frac{\operatorname{vol}(B(p,r))}{\omega_nr^n}.
    \]
The above limit is independent of the choice of the base point $p\in M$. When $\beta>0$, $M$ is said to have Euclidean volume growth.

Locally conformally flat manifolds with nonnegative Ricci curvature have been completely classified by Zhu \cite{Z} and Carron--Herzlich \cite{CH}. According to this classification, one essential case to study is when $(M^n, g)$ takes the form $(\mathbb{R}^n, g=e^{2w}g_0)$, for which Ma and Qing \cite{MQ} developed a fundamental theory on the asymptotic exponent of the conformal factor $w$.

We now explain in some details of Ma-Qing's theory. There exists $m\in [0,1]$ such that 
\[
    \liminf_{x\to \infty} \frac{w(x)}{\ln |x|}=-m.
\]
Moreover, there exists an exceptional set $E$, which is $n$-thin at infinity (in the sense of $n$-potential theory) such that 
\[
\lim_{x\to \infty, x\notin E} \frac{w(x)}{\ln |x|}=-m.
\]

In a recent breakthrough~\cite{M}, Ma refined the above estimate to the following: setting $\underline{w}(r)=\inf_{|x|=r} w(x)$, there exists a strong  $\mathcal{E}$-set $E$, such that 
\begin{equation} \label{eq:asympradial}
w(x)=\underline{w}(|x|)+o(1), \quad x\to \infty \quad \text{and} \quad x\notin E.
\end{equation}
Using this refined radial symmetry of $w$, Ma \cite{M} showed that 
\begin{equation}\label{eq. volume-ratio formula}
    \beta^{\frac{1}{n-1}}=1-m.
\end{equation}

Furthermore, Ma computed the exact values of the scale-invariant integral of scalar curvature, $\lim_{r\to \infty} \frac{1}{r^{n-2}}\int_{B_r} R_g \dd V_g$, which affirmatively answers a question of Yau on the finiteness of scalar curvature for complete manifolds with nonnegative Ricci curvature in the locally conformally flat setting.

Building on this asymptotic analysis of the conformal factor, we establish our main result by treating the two cases $\beta=0$ and $\beta>0$ separately.

\begin{theoremx} \label{thm: T1}
    Let $g=e^{2w}g_{0}$ be a smooth and complete metric on $\mathbb{R}^n$ with nonnegative Ricci curvature, suppose $\beta=0$, then for every finite $d\ge0$,
\[
 \mathcal{H}_d(\R^n,g)=\R.
\]
\end{theoremx}

{For $\beta>0$, in view of (\ref{eq. volume-ratio formula}), we have $m<1$. Set $a=1-m\in (0,1]$ and define 
\[
\mathcal{C}_a=
    \bigl([0,\infty)\times\mathbb S^{n-1},
    d\rho^2+a^2\rho^2g_{\mathbb S^{n-1}}\bigr).
\]}
\begin{theoremx} \label{thm:T2}
        Let $g=e^{2w}g_{0}$ be a smooth and complete metric on $\mathbb{R}^n$ with nonnegative Ricci curvature, suppose $\beta\in (0,1)$ , then for every integer $d\ge1$,
\[
 \dim \mathcal{H}_d(\R^n,g)=\dim \mathcal{H}_d(\mathcal{C}_a) \leq \dim \mathcal{H}_{d-1}(\mathbb{R}^n, g_0)<\dim \mathcal{H}_{d}(\mathbb{R}^n, g_0).
\]
\end{theoremx}

In spirit, the study in this paper can be regarded as an extension of the Li--Tam's result~\cite{LT2} to higher dimension. As a precursor for the main result, we first need to  establish a comparison between  the distance induced by the conformal metric $g=e^{2w}g_0$ and the underlying Euclidean distance.

\begin{theorem} \label{thm:distance}
Let $g=e^{2w} g_0$ be a conformal metric on $\mathbb{R}^n$ with nonnegative Ricci curvature and $\beta>0$. Let $r(x)$ denote the distance function to the origin induced by $g$, then
\begin{equation} \label{distancecompare}
\lim_{x\to \infty} \frac{\ln r(x)}{\ln |x|}=1-m. 
 \end{equation} 
\end{theorem}

The proof of Theorem \ref{thm:distance} relies on Ma-Qing's theory. With new advances by Ma, the relation (\ref{distancecompare}) even holds when $\beta=0$ and can be implicitly deduced from Ma's argument. 

{Distance comparison enables us to translate geometric and function-theoretic information into Euclidean terms. Then the asymptotic analysis \cref{eq:asympradial} helps to establish a bijection, through a blow-down procedure,
between $\mathcal{H}_d(\mathbb{R}^n, g)$ and $\mathcal{H}_{d}(\mathcal{C}_a)$.
The difficulty arises from the presence of a strong $\mathcal{E}$-set $E$ in \cref{eq:asympradial} and the lack of uniform ellipticity for the limiting harmonic equations, which prevent convergence in the standard $G$-convergence framework \cite{ZKON}. To circumvent this, we devise a weighted $G$-convergence scheme and then carry through the blow-down argument and the dimension estimates.}

The organization of the paper is as follows. In Section 2, we prove Theorem \ref{thm: T1}. In Section 3, we provide a proof for \cref{thm:distance} even though this result is now subsumed by ~\cite{M}. The proof does not need to appeal to Ma's new theory and has its own independent interest. In Section 4, we discuss the blow-down scheme and weighted $G$-convergence. In Section 5, we prove Theorem \ref{thm:T2}. In the last section, dealing with remaining cases, we conclude the proof of Theorem \ref{thm:A}.

{\bf Declaration on the use of AI:}
The authors  obtained \cref{thm:distance} and then developed the overall strategy for the proofs of Theorem~\ref{thm:T2} 
---
the main body of Theorem~\ref{thm:A}
---
prior to any involvement of AI in the proofs.
AI was subsequently used to assist in generating certain technical details.
All arguments were written by the authors.
We take full responsibility  for the content of this manuscript.

\section{The case with vanishing asymptotic volume ratio}

Kasue and Carron proved Liouville theorems under sublinear diameter growth of geodesic spheres~\cite{K1,K2,C}. Carron's  formulation \cite[Theorem 1.1]{C} is particularly convenient here.
\begin{proposition} \label{prop:diam}
Let $(M^n,g)$ be complete with $\operatorname{Ric}_g\ge0$.  Assume that, for some $p\in M$,
\begin{equation}\notag
 \operatorname{diam}_g\bigl(\partial B_R^g(p)\bigr)=o(R).
\end{equation}
Then every harmonic function of polynomial growth of finite degree in $M$ must be constant.
\end{proposition}

\begin{proof}[Proof of Theorem \ref{thm: T1}]
Adopting the notation in \cite{M}, i.e., set $\underline{w}(r)=\inf_{|x|=r} w(x)$, in view of the assumption that $\beta=0$, 
it follows that $\lim_{r\to \infty} \frac{\underline{w}(r)}{\ln r}=- 1$. Consider the radial metric $\underline{g}=e^{2\underline{w}} g_0$. Introducing the intrinsic radial coordinates
$\rho(r)=\int_0^r e^{\underline{w}(s)}\dd s$ and the warping function
$ a(\rho)=r e^{\underline{w}(r)},$
we find 
\begin{equation}\notag
 \underline{g}
 =d\rho^2+a(\rho)^2g_{\mathbb{S}^{n-1}}.
\end{equation}

Based on the asymptotic radial symmetry (\ref{eq:asympradial}), Ma \cite[Page 29]{M} showed that 
\begin{equation} \label{eq:inclusion}
B^{\underline{g}}_R\subset B^{g}_R\subset B^{\underline{g}}_{R+o(R)},
\end{equation}
Combining this with the volume growth assumption $\beta=0$ and $w\geq \underline{w}$, we get 
\begin{align*}
    0=\lim_{R\to+\infty}
    \frac{vol_g(B^g(p,R))}{|\mathbb{B}^n|R^n}
    &\geq \lim_{R\to+\infty}
    \frac{vol_{\underline{g}}(B^{\underline{g}}(p,R))}{|\mathbb{B}^n|R^n} \\ \notag
    &=\lim_{R\to+\infty}
    \frac{area_{\underline{g}}
    (\partial B^{\underline{g}}(p,R))}{|\mathbb{S}^{n-1}|R^{n-1}} 
    =\lim_{\rho\to+\infty}
    \frac{a^{n-1}(\rho)}{\rho^{n-1}}
\end{align*}

It follows that
\begin{equation}\label{eq:a-small}
 a'(\rho)\to0,
 \qquad
 \frac{a(\rho)}{\rho}\to0,\qquad 
 \text{as }\rho\to+\infty.
\end{equation}
Therefore, 
\[
\operatorname{diam}(\partial B^{\underline{g}}_R)\leq \pi a(R)=o(R).
\]
By \cref{eq:inclusion}, one easily see that $\operatorname{diam}(\partial B^{g}_R)=o(R)$ as well. Hence the conclusion follows from \cref{prop:diam}.
\end{proof}

We also include a simple ODE argument for radially symmetric conformal metric $(\mathbb{R}^n, g=e^{2w(|x|)}g_0)$ with nonnegative Ricci curvature and $\beta=0$.

\begin{theorem}\label{thm:radial}
Let
\[
 g=e^{2w(|x|)}g_{0}
\]
be a complete and smooth rotationally symmetric metric on $\R^n$ with nonnegative Ricci curvature. Assume
$\beta=0$, then for every $d\ge0$
\[
 \mathcal{H}_d(\R^n,g)=\R.
\]
\end{theorem}

\begin{proof}
As above, we can now assume 
$g=d\rho^2+a(\rho)^2g_{\mathbb{S}^{n-1}}$ under the intrinsic radial coordinates and $a(\rho)$ satisfies (\ref{eq:a-small}).

Let $u$ be a smooth harmonic function of polynomial growth.  Expand it in spherical harmonics:
\[
 u(\rho,\theta)
 =\sum_{\ell=0}^\infty\sum_m
 v_{\ell,m}(\rho)Y_{\ell,m}(\theta),
\]
where
\[
 -\Delta_{\mathbb{S}^{n-1}}Y_{\ell,m}
 =\lambda_\ell Y_{\ell,m},
 \qquad
 \lambda_\ell=\ell(\ell+n-2).
\]
The harmonicity of $u$ yields that every coefficient $v_{\ell,m}$ satisfies an ODE
\begin{equation}\label{eq:mode-ode}
 v''+(n-1)\frac{a'}a v'
 -\frac{\lambda_\ell}{a^2}v=0.
\end{equation}

For $\ell=0$, equation \cref{eq:mode-ode} gives
\[
 (a^{n-1}v')'=0.
\]
The regularity at the origin forces $v'=0$, so the radial mode $ v_{0,1}Y_{0,1}$ is constant since the only harmonic function on $\mathbb{S}^{n-1}$ is constant. 

Now suppose $\ell\ge1$.  A nonzero coefficient that is regular at the origin has, after changing its sign if necessary,
\[
 v(\rho)>0,
 \qquad
 v'(\rho)>0
\]
for all $\rho>0$.  Indeed, the divergence form
\[
 (a^{n-1}v')'=\lambda_\ell a^{n-3}v
\]
prevents a positive regular solution from acquiring a first zero or a first nonpositive derivative.

Define
\[
 s(\rho)=\int_{\rho_0}^{\rho}\frac{dt}{a(t)}.
\]
In the $s$-variable, \cref{eq:mode-ode} becomes
\[
 v_{ss}+(n-2)a'(\rho(s))v_s-\lambda_\ell v=0.
\]
Set $q=v_s/v>0$.  Then
\begin{equation}\label{eq:riccati}
 q'=\lambda_\ell-(n-2)a' q-q^2.
\end{equation}
Since $\beta=0$, we could argue as before to derive 
\begin{equation}\label{eq:a-small_radial case}
    0=\lim_{R\to+\infty}
    \frac{vol_{g}(B^{g}(p,R))}{|\mathbb{B}^n|R^n}
    =\lim_{R\to+\infty}
    \frac{area_{g}
    (\partial B^{g}(p,R))}{|\mathbb{S}^{n-1}|R^{n-1}}
    =\lim_{\rho\to+\infty}
    \frac{a^{n-1}(\rho)}{\rho^{n-1}}
    =\lim_{\rho\to+\infty}a'(\rho)^{n-1}.
\end{equation}

Choose $c>0$ and $S$ sufficiently large so that
\[
 c^2+(n-2)a'(\rho(s))c<\lambda_\ell
 \qquad \forall s\ge S.
\]
Equation \cref{eq:riccati} shows that $q$ cannot cross the level $c$ downward for large $s$, and if $q<c$ for a long interval, then $q'$ is uniformly positive. Hence, after increasing $S$, we have
\[
 q(s)\ge c
 \qquad \forall s\ge S.
\]
Therefore,
\begin{equation}\label{eq:exp-s}
 v(\rho)\ge C\exp\bigl(c s(\rho)\bigr).
\end{equation}
Finally, \cref{eq:a-small_radial case} implies that
\begin{equation}\label{eq:s-superlog}
 \frac{s(\rho)}{\log\rho}\to\infty.
\end{equation}
Indeed, for every $\varepsilon>0$, one has $a(t)\le\varepsilon t$ for all sufficiently large $t$, and hence
\[
 s(\rho)\ge\frac1\varepsilon\log\frac{\rho}{\rho_\varepsilon}.
\]
Combining \cref{eq:exp-s} and \cref{eq:s-superlog}, every nonzero mode with $\ell\ge1$ grows faster than every power of $\rho$.  Since $u$ has polynomial growth, all such coefficients vanish.  
In conclusion, only the $\ell=0$ mode remains and $u$ must be a constant.
\end{proof}

\section{Distance comparison for Euclidean volume growth}

We carry out the proof of Theorem \ref{thm:T2} in the following three sections. The general strategy is inspired by Lin's work on asymptotically conic elliptic operators \cite{L}. 
In this section, we compare the intrinsic distance of the conformal metric with the Euclidean distance, thereby translating intrinsic polynomial growth into growth estimates adapted to Euclidean rescaling. 
In Section~4, we establish a weighted compactness theorem for the rescaled harmonic equations and show that normalized blow-down sequences converge to harmonic functions on a fixed metric cone.
In Section~5, we exploit the spectral decomposition  on the limiting cone to obtain a uniform $L^2$ growth ratio upper bound. A scale-dependent Gram-Schmidt procedure then allows us to blow down a finite-dimensional space of harmonic functions simultaneously while preserving linear independence, leading to the desired dimension estimate.

In this section, we shall prove \cref{thm:distance}.  Two fundamental constants that we frequently refer to are: 
\begin{itemize}
    \item $\beta$: the asymptotic volume ratio of a complete noncompact manifold with nonnegative Ricci curvature, i.e., 
    \begin{equation*}
        \beta=\lim_{r\to \infty}\frac{\operatorname{vol}(B(p,r))}{\omega_nr^n}.
    \end{equation*}
    \item $m$: the asymptotic exponent of the conformal factor. Let $(\mathbb{R}^n, e^{2w}g_0),(n\geq 3)$ be a complete Riemannian manifold with nonnegative Ricci curvature, then there exists $m\in [0,1]$ (see ~\cite[Theorem 1.3]{MQ}), such that
    \begin{equation}\label{def of m}
        \liminf_{x\to \infty} 
    \frac{ w(x)}{\ln |x|}=-m.
    \end{equation}  
\end{itemize}

We shall first derive a technical tool: a refined Harnack inequality \cref{proposition. Harnack ineq.}. Then we divide the proof into two steps, following the strategy of Li-Tam~\cite{LT2}. The first step appeals to the pointwise lower bound of the conformal factor $w$ due to Ma-Qing, and a lower bound of the $g$-volume of the Euclidean ball is obtained. Then an elementary inclusion relation between the geodesic ball and the Euclidean ball implies that
    \[
        \liminf_{x\to \infty}\frac{\ln r(x)}{\ln |x|}\geq 1-m.
    \]
The refined Harnack inequality is used in this part.

The difficulty of the second step lies in the fact that there is no pointwise upper bound of the conformal factor $w$. In contrast, inspired by \cite[Theorem 3.1, Proposition 4.1]{MQ}, an upper bound of the $g$-volume of the Euclidean ball could be derived via several involved results in nonlinear potential theory. It follows analogously that
    \[
        \limsup_{x\to \infty}\frac{\ln r(x)}{\ln |x|}\leq 1-m.
    \]
Once again, the refined Harnack inequality \cref{proposition. Harnack ineq.} is used in this step.

We also include a proof relating the the comparison of two distance functions to the asymptotic volume ratio $\beta$. The conformal invariance of the $n$-Laplace operator indicates that $\ln |x|$ is the $n$-Green function on $(M^n,g)$. Meanwhile, $\int_1^{r(x)}A(t)^{-\frac{1}{n-1}}\dd t$ is the $n$-Green function when $(M^n,g)$ is radially symmetric. Combining with the refined  Harnack inequality, a comparison principle argument is applied to these functions to show that 
    \[
        \limsup_{x\to\infty}\frac{\ln r(x)}{\ln |x|}\leq \beta^{\frac{1}{n-1}}.
    \]

\begin{remark}
In view of \cref{eq:inclusion}, the conclusion of \cref{thm:distance} follows. Moreover it also holds when $\beta=0$. We present our proof anyway as it has independent interest. 
\end{remark}

\subsection{A refined Harnack inequality}
 \ \\
In this subsection, we prove a refined Harnack inequality which serves as a basis analytical tool in each step.

\begin{proposition}\label{proposition. Harnack ineq.}
    Let $(\mathbb{R}^n,e^{2w}g_0)$ be a complete Riemannian manifold with nonnegative Ricci curvature and $\beta>0$.
    Then 
    \[
        s(R)\leq i(R) e^{o(1)} \quad \text{as }R\to +\infty,
    \]
    where $s(r):=\sup_{\partial B^g(0,r) }\ln |x| $ and $i(r):=\inf_{\partial B^g(0,r)}\ln |x| $.
\end{proposition}

We first establish several simple lemmas.  
\begin{lemma}
     Let $(\mathbb{R}^n,e^{2w}g_0)$ be a complete Riemannian manifold with nonnegative Ricci curvature. Denote $r(x):=\operatorname{dist}_g(x,0)$ the distance to the origin.
     Then 
     \[
         -\Delta_n \ln r(x)\geq 0.
     \]
\end{lemma}
\begin{proof}
    The proof is by a direct calculation.
    \begin{align*}
        -\Delta_n^g \ln r
        &=-\mathrm{div}^g
        (|\nabla^g \ln r|_g^{n-2}\nablag \ln r)
        =-\mathrm{div}^g(r^{1-n}\nablag r)\\
        &=r^{1-n}\left(-\Delta^g r+\frac{n-1}{r}\right)\geq 0.
    \end{align*}
\end{proof}

\begin{lemma}\label{lemma. nondecreasing of i(r)}
    Let $(\mathbb{R}^n,e^{2w}g_0)$ be a complete Riemannian manifold with nonnegative Ricci curvature. 
     Then  $s(r):=\sup_{\partial B^g(0,r)}\ln |x|,\   i(r):=\inf_{\partial B^g(0,r)}\ln |x|$ are both non-decreasing functions as $r\to +\infty$.
\end{lemma}

\begin{proof}
    By the conformal invariance of the $n$-Laplace operator, we have
    \[
        \Delta_n^g\ln |x|=0,\quad \text{on } \mathbb{R}^n\setminus\{0\}.
    \]
    Note that $\ln |x|\to-\infty$ as $x\to 0$, we could apply the maximum principle to $\ln |x|$ on $B^g(0,R_2)\setminus B^g(0,r)$, where $0<r\ll1,\ r<R_2$, to conclude that
    \[
        \sup_{\partial B^g(0,R_1)}\ln |x|
        \leq  \sup_{\partial B^g(0,R_2)}\ln |x|,\quad \forall \,0<R_1<R_2.
    \]
 Similarly, noting that $\ln |x|\to +\infty$ as $x\to \infty$, we apply the minimum principle to $\ln |x|$ on $B^g(0,R)\setminus B^g(0,r_1)$, where $0<r_1<R,\ 1\ll R$, to conclude that
    \[
        \inf_{\partial B^g(0,r_2)}\ln |x|
        \geq  \inf_{\partial B^g(0,r_1)}\ln |x|,\quad \forall \,0<r_1<r_2.
    \]
\end{proof}

\begin{lemma}\label{lemma. weak Harnack ineq. for grad G}
    Let $(\mathbb{R}^n,e^{2w}g_0)$ be a complete Riemannian manifold with nonnegative Ricci curvature. Define $G(x):=\ln |x|$ and $f(x)=(n-1)^2\frac{|\nabla^g G|_g^2}{G^2}$ on $\mathbb{R}^n\setminus B(0,e)$. Then for $B^g(p,R)\subset \mathbb{R}^n\setminus B(0,e)$, the following holds.
    \begin{enumerate}
        \item There exist constants $b_1=b_1(n)\geq n^3$ and $C=C(n)>0$ such that
        \[
            |f|_{L^{\infty}(B^g(p,\frac{R}{2}))}
            \leq C(n)
            \left(
            \fint_{B^g(p,\frac{3R}{4})}f^{b_1}\dd V_g
            \right)^{\frac{1}{b_1}},
        \]
        where $\fint_{B^g(p,R)}f\dd V_g:=\frac{1}{Vol_g(B^g(p,R))}\int_{B^g(p,R)}f\dd V_g.$
        \item There exists a constant $C(n)$  such that
        \[
            |f|_{L^{\infty}(B^g(p,\frac{R}{2}))}
            \leq C(n)
            \left(
            \fint_{B^g(p,\frac{3R}{4})}f^{\frac{n}{2}}\dd V_g
            \right)^{\frac{2}{n}}.
        \]
    \end{enumerate}
\end{lemma}
\begin{proof}
    By conformal invariance of the $n$-Laplace operator, there holds
    \[
        \Delta_n^g G=\Delta_n G=0,\quad \text{in } \mathbb{R}^n\setminus B(0,e).
    \]

    Define $u:=-(n-1)\ln G$, by a straightforward computation, it follows
    \[
        \mathrm{div}^g(|\nablag u|_g^{n-2}\nablag u)=|\nablag u|_g^n.
    \]
    Moreover, notice that $f=|\nabla^g u|_g^2$,  a Moser iteration argument as that in \cite[Theorem 1.1]{WZ} would derive the first assertion   verbatim (see \cite[(2.18)]{WZ}). 

    Now we shall deduce the second assertion.
    We claim that for any $\sigma\in (0,1)$, there holds
    \[
        |f|_{L^{\infty}(B^g(y,\sigma R))}\leq 
        \frac{C(n)}{(1-\sigma)^2}
        \left(
            \fint_{B^g(y,R)}f^{\frac{n}{2}}\dd V_g
            \right)^{\frac{2}{n}},
            \quad \text{provided } B^g(y,R)\subset \mathbb{R}^n\setminus B(0,e),
    \]
    from which the assertion $(2)$ follows.

    In fact, fix some $B^g(y,R)\subset \mathbb{R}^n\setminus B(0,e)$. Notice that $B^g(x,(1-\sigma)R)\subset B^g(y,R)\subset \mathbb{R}^n\setminus B(0,e)$ if $x\in B^g(y,\sigma R)$, the assertion $(1)$ implies
    \begin{align*}
        |f(x)|&\leq C(n) \left(
            \fint_{B^g(x,\frac{3}{4}(1-\sigma)R)}f^{b_1}\dd V_g
            \right)^{\frac{1}{b_1}}\\
            &\leq C(n)\left(
            \frac{\operatorname{vol}_g(B^g(y,R))}{\operatorname{vol}_g(B^g(x,\frac{3}{4}(1-\sigma)R))}
            \fint_{B^g(y,R)}f^{b_1}\dd V_g
            \right)^{\frac{1}{b_1}}\\
            &=C(n)\left(
            \frac{\operatorname{vol}_g(B^g(y,R))}{\operatorname{vol}_g(B^g(x,(1+\sigma)R)}\,
            \frac{\operatorname{vol}_g(B^g(x,(1+\sigma)R)}{\operatorname{vol}_g(B^g(x,\frac{3}{4}(1-\sigma)R))}
            \fint_{B^g(y,R)}f^{b_1}\dd V_g
            \right)^{\frac{1}{b_1}}\\
            &\leq C(n)\frac{1}{(1-\sigma)^{\frac{n}{b_1}}}\left(
            \fint_{B^g(y,R)}f^{b_1}\dd V_g
            \right)^{\frac{1}{b_1}},
    \end{align*}
    where in the last line, we used the fact that $B^g(y,R)\subset B^g(x,(1+\sigma)R)$ and the Bishop-Gromov volume comparison theorem.
    It follows from Young's inequality that
    \begin{align*}
        \sup _{B^g(y,\sigma R)}|f|
        &\leq \frac{C(n)}{(1-\sigma)^{\frac{n}{b_1}}} \left(
            \fint_{B^g(y,R)}f^{b_1}\dd V_g
        \right)^{\frac{1}{b_1}}\\
        &\leq \frac{C(n)}{(1-\sigma)^{\frac{n}{b_1}}}
        (\sup_{B^g(y,R)}|f|)^{1-\frac{n}{2b_1}}
        \left(
            \fint_{B^g(y,R)}f^{\frac{n}{2}}\dd V_g
        \right)^{\frac{1}{b_1}}\\
        &\leq \frac{1}{2}\sup_{B^g(y,R)} |f|
        +\frac{n}{2b_1}
        \left(2-\frac{n}{b_1}\right)
        ^{\frac{2b_1}{n}(1-\frac{n}{2b_1})}
        \frac{C(n)}{(1-\sigma)^2}
        \left(
            \fint_{B^g(y,R)}f^{\frac{n}{2}}\dd V_g
        \right)^{\frac{2}{n}}.
    \end{align*}
    Then a standard hole filling iteration lemma (see, for example, \cite[Lemma 3]{BC}) implies that 
    \[
        \sup_{B^g(y,\sigma R)}|f|\leq \frac{C(n)}{(1-\sigma)^2}
        \left(
            \fint_{B^g(y,R)}f^{\frac{n}{2}}\dd V_g
        \right)^{\frac{2}{n}}.
    \]
\end{proof}

\begin{proof}[Proof of \cref{proposition. Harnack ineq.}]
    Define $G(x):=\ln |x|$. Consider $f\:=|\nabla^g \ln G|_g^2$.
    By the conformal invariance of the $n$-Dirichlet energy in $n$-dimensional space, there holds
    \begin{align*}
        \int_{\mathbb{R}^n\setminus B(0,e)}f^{\frac{n}{2}}\dd V_g
        &=\int_{\mathbb{R}^n\setminus B(0,e)} |\nabla^g \ln G|_g^n\dd V_g
        =\int_{\mathbb{R}^n\setminus B(0,e)} |\nabla \ln G|^n\dx\\
        &=|\mathbb{S}^{n-1}|\int_e^{+\infty}\frac{1}{t(\ln t)^n}\dd t
        =|\mathbb{S}^{n-1}|\int_1^{+\infty}\frac{1}{t^n}\dd t\\
        &=\frac{|\mathbb{S}^{n-1}|}{n-1}<+\infty.
    \end{align*}
    Therefore, there exists a function $\eta(r)$ satisfying $0\leq \eta(r)\leq \frac{|\mathbb{S}^{n-1}|}{n-1}$, $\eta(r)\to0$ as $r\to +\infty$ and
    \begin{equation}\label{ineq. smallness of integral of f to n/2}
        \int_{\mathbb{R}^n\setminus B^g(0,r)} f^{\frac{n}{2}}\dd V_g
        \leq \eta(r).
    \end{equation}
    It follows from \cref{lemma. weak Harnack ineq. for grad G} and $\beta>0$ that, for 
    $ B^g(p, 2r)\subset \mathbb{R}^n\setminus B(0,e)$,
    \begin{align}\label{ineq. upper bound of gradient of ln G}
        \sup_{B^g(p,r)}|\nabla^g \ln G|_g
    =(\sup_{B^g(p,r)} f)^{\frac{1}{2}}
    &\leq C(\beta,n) r^{-1} (\int_{B^g(p,\frac{3}{2}r)} f^{\frac{n}{2}}\dd V_g )^{\frac{1}{n}}.
    \end{align}

    For fixed $r\gg 1$, assume $x_0\in \partial B^g(0,r)$ satisfying $G(x_0)=s(r)$, and $y_0\in \partial B^g(0,\frac{r}{2})$ satisfying $G(y_0)=i(\frac{r}{2})$.
 By \cite[Proposition 4.5]{HK}, there exists a constant $C_0=C_0(n,\beta)$, such that $B^g(0,r)\setminus B^g(0,\frac{r}{2})$ is connected in $B^g(0,C_0r)\setminus B^g(0,\frac{r}{2C_0})$. Then by Bishop-Gromov volume comparison theorem, there exists some constant $N=N(n,C_0)$ such that there are at most $N$ disjoint geodesic balls each with radius $\frac{r}{4}$ contained in $B^g(0,C_0r)\setminus B^g(0,\frac{r}{2C_0})$. Therefore, there exists a curve $\gamma_0$ and  ($N+2$) geodesic balls $B^g(p_i,\frac{r}{4})
    \subset B^g(0,C_0r)\setminus B^g(0,\frac{r}{2C_0})$, such that $\gamma_0\subset \cup_{i=1}^{N+2}B^g(p_i,\frac{r}{4})$ and $\mathrm{Length}(\gamma_0)\leq \frac{r}{2}(N+2)$.
    Then we get from \cref{ineq. upper bound of gradient of ln G} and  \cref{ineq. smallness of integral of f to n/2} that
    \begin{align*}
        \ln s(r)-\ln i\left(\frac{r}{2}\right)
        &=\ln G(x_0)-\ln G(y_0) \\ \notag
        &\leq \int_{\gamma_0}|\nabla^g\ln G|_g\dd l_{\gamma_0}                \leq \sum_{i=1}^{N+2}\frac{r}{2}\sup_{B^g(p_i,\frac{r}{4})}|\nabla^g \ln G|_g\\
        &\leq \frac{r}{2}C(\beta,n)r^{-1}\sum_{i=1}^{N+2}
    \left(\int_{B^g(p_i,\frac{3}{8}r)} f^{\frac{n}{2}}\dd V_g \right)^{\frac{1}{n}}\\
    &\leq C(\beta,n)(N+2)
    \left(\int_{B^g(0,(C_0+1)r)\setminus B^g(0,\frac{1}{2(C_0+1)}r)} f^{\frac{n}{2}} \dd V_g \right)^{\frac{1}{n}}\\
    &\leq C(\beta,n)\, \eta^{\frac{1}{n}}\left(\frac{r}{2(C_0+1)}\right).
    \end{align*}
    Combined with \cref{lemma. nondecreasing of i(r)} and the decay of $\eta(r)$, the above estimate yields that
    \[
        s(r)\leq i\left(\frac{r}{2}\right)e^{C\eta^{\frac{1}{n}}
        \left(\frac{r}{2(C_0+1)}\right)}
        \leq i(r)e^{o(1)},\quad \text{as } r\to +\infty.
    \]
\end{proof}
\begin{remark}
    Recall that the Harnack inequality for a nonnegative $n$-harmonic function $u$ on $B(0,R)\subset\mathbb{R}^n$ states as: there exists a constant $C(n)$ such that
    \[
        \sup_{B(0,R/2)}u\leq C(n)\inf_{B(0,R/2)}u,
    \]
    see, for example, \cite[Theorem 2.20]{Lin}.
    In contrast, \cref{proposition. Harnack ineq.} claims an asymptotically sharp constant $e^{o(1)}$ for the $n$-harmonic function $\ln |x|$ on the conformal flat manifold $(\mathbb{R}^n,e^{2w}g_0)$ with nonnegative Ricci curvature.
\end{remark}

\subsection{Lower bound by $m$}
\begin{proposition}\label{proposition. second part}
    Let $(\mathbb{R}^n,e^{2w}g_0), (n\geq 3)$ be a complete Riemannian manifold with nonnegative Ricci curvature and $\beta>0$.
    Then 
    \[
        \liminf_{x\to \infty}\frac{\ln r(x)}{\ln |x|}\geq 1-m.
    \]

\end{proposition}
\begin{proof}   
There is nothing to prove if $m=1$. So we assume $0\leq m<1$. $C$ is a constant which may be different from line to line and is independent of $r$.
    
Define $\underline{\rho}(r):=\inf_{\partial B^g(0,r) }|x|=e^{i(r)}$, where $i(r):=\inf_{\partial B^g(0,r)}\ln |x| $. It is easy to see that 
    $B(0,\underline{\rho}(r))\subset B^g(o,r)$. Then it follows that
    \begin{equation}\label{eq.subset volume relation}
        \operatorname{vol}_g(B(0,\underline{\rho}(r))
        \leq \operatorname{vol}_g(B^g(0,r)).
    \end{equation}

    On the one hand, by \cite[Theorem 1.3]{MQ}, there exist some constants C and $\rho_0$ such that
    \begin{equation}\label{ineq. pointwise lower bound of phi}
        w(x)\geq -m\ln |x|-C,\quad \forall \ |x| \geq \rho_0.
    \end{equation}
    Therefore, for $\rho(r)> \rho_0$ we have
    \begin{align*}
        \operatorname{vol}_g
        (B(0,\underline{\rho}(r))
        &=\int_{B(0,\underline{\rho}(r))\setminus B(0,\rho_0)} e^{nw(x)}\dx
        +\int_{B(0,\rho_0)} e^{nw(x)}\dx\\
        &\geq |\mathbb{S}^{n-1}|\int_{\rho_0}^{\underline{\rho}(r)}C t^{-mn}t^{n-1}\dd t
        \\
        &\geq
            C\underline{\rho}(r)^{n(1-m)}-C \quad\quad(\ \text{since } 0\leq m<1), 
    \end{align*}
    
    On the other hand, by Bishop-Gromov volume comparison, it follows that
    \[
        \operatorname{vol}_g(B^g(0,r))\leq Cr^n,\quad \forall r>0.
    \]
    Insert these into \cref{eq.subset volume relation} to get
    \[
        C\underline{\rho}(r)^{n(1-m)}-C\leq C r^n.
    \]
    Therefore,
    \[
        \frac{n\ln r+C}{i(r)}\geq \frac{\ln(C\underline{\rho}(r)^{n(1-m)}-C)}{i(r)}=\frac{\ln (e^{n(1-m)i(r) }-C)+C}{i(r)}
    \]
Taking the limit inferior, we obtain
    \[
        \liminf_{r\to \infty}\frac{\ln r}{i(r)}\geq 1-m.
    \]

Finally we apply \cref{proposition. Harnack ineq.} to complete the proof as follows:
    \[
        \liminf_{x\to\infty}\frac{\ln r(x)}{\ln |x|}\geq 
        \liminf_{r\to \infty}\frac{\ln r}{s(r)}
        \geq \liminf_{r\to \infty}\frac{\ln r}{i(r)e^{o(1)}}\geq 1-m.
    \]
\end{proof}

\subsection{Upper bound by $m$ }
\begin{proposition}\label{proposition. third part}
    Let $(\mathbb{R}^n,e^{2w}g_0), (n\geq 3)$ be a complete Riemannian manifold with nonnegative Ricci curvature and $\beta>0$.
    Then 
    \[
        \limsup_{x\to \infty}\frac{\ln r(x)}{\ln |x|}\leq 1-m.
    \]
\end{proposition}

Since there is no pointwise upper bound of the conformal factor $w$ like the lower bound \cref{ineq. pointwise lower bound of phi}, the argument in \cref{proposition. second part} could not be applied directly to  prove \cref{proposition. third part}.
Instead, inspired by Ma-Qing's work on the asymptotic upper bound of $w$ out of some exceptional set (see \cite[Theorem 3.1]{MQ}), we shall utilize several deep results in nonlinear potential theory to derive an upper bound of the $g$-volume of Euclidean balls.

We start with the main technical lemma, whose main part is a modification of \cite[proposition 4.1]{MQ}.

\begin{lemma}\label{lemma. integral of exp(n phi)}
    Let $(\mathbb{R}^n,e^{2w}g_0)$ be a complete Riemannian manifold with nonnegative Ricci curvature.
    Then for any $\epsilon>0$, there exist some constants $R_0>0$ and $C=C(n)>0$ such that
    \[
        \int_{B(0,2R)\setminus B(0,R)} e^{nw(x)}\dx 
        \leq C(n)R^{n(1-m)+\epsilon},\quad \forall R\geq R_0.
    \]
\end{lemma}

\begin{proof}
Let $w_0(y):=w(\frac{y}{|y|^2})-2\ln|y|$, it follow that $w_0\in C^{\infty}(\mathbb{R}^n\setminus\{0\})$. Then $g=e^{2w(x)}|\dx|^2=e^{2w_0(y)}|\dy|^2 $, where $x=\frac{y}{|y|^2}$. By \cite[(2.18)]{MQ}, it follows that
    \begin{equation}\label{eq. equation of w}
        -\Delta_n w_0
        =\mathrm{Ric}_g
        \left(\frac{\nabla^g w_0}{|\nabla^g w_0|_g},\frac{\nabla^g w_0}{|\nabla^g w_0|_g}\right)
            |\nabla w_0|^{n-2}e^{2w_0}=:f
            \quad \text{on } \mathbb{R}^n\setminus\{0\}.
    \end{equation}
    
    Since $g=e^{2w_0}|\dy|^2 $ is complete at the origin and its scalar curvature is nonnegative, we know from
    \cite[Proposition 8.1]{CHY04} that
    \[
        \lim_{y\to 0} w_0(y)=+\infty.
    \]
    Then, from \cite[proposition 1.1]{B-V89} (see also \cite[Theorem 2.4]{MQ}), we know that $f\in L^1(B(0,2))$ and there exists a constant $\delta\geq 0$ such that 
    \[
        -\Delta_n w=f+\delta\delta_0 \quad \text{on } B(0,2)
    \]
    in the distributional sense.
    Therefore we could define a nonnegative Radon measure $\mu$ as $\mu:=\mu_f+\delta \delta_0$, where
    \[
        \mu_f(\Omega):=
        \int_{\Omega}f\dx,\quad \forall\, \Omega\subset  B(0,2).
    \]
    Then we could interpret \cref{eq. equation of w} as
    \[
        -\Delta_n w_0=\mu\geq 0 \quad \text{on } B(0,2).
    \]
    Hence we could apply \cite[Lemma 3.7, Theorem 3.1]{MQ} and obtain that, there exists a constant $m_1$ such that, for any fixed $\alpha\in (0,1)$, there holds
    \begin{equation}\label{eq. lemma 3.7 of MQ}
        \lim_{y\to 0}\frac{\inf_{B(y,\alpha|y|)}w_0}{\ln\frac{1}{|y|}}
        =m_1.
    \end{equation}
    It turns out that $m_1=2-m\in [1,2]$ (see \cite[(5.5),(5.6)]{MQ}). 

    Recall the Wolff potential
    \[
        W_{1,n}^{\mu}(x,r):=\int_0^r \mu(B(x,t))^{\frac{1}{n-1}}\frac{1}{t}\dd t.
    \]
    For any fixed $y\in B(0,1)$,  apply the foundational estimates \cite[Theorem 1.6]{KM94} (see also \cite[Theorem 2.6]{MQ}) to $w-\inf_{B(y,\frac{3}{4}|y|)}w$ in $B(y,\frac{3}{4}|y|)$, we obtain: there exists some constants $C_1(n)>0,\ C_2(n)>0$ such that
    \begin{equation}\label{ineq. KM94 upper bound}
        \frac{w_0(y)}{\ln\frac{1}{|y|}}\leq
        \frac{\inf_{B(y,\frac{3}{4}|y|)}w_0}{\ln\frac{1}{|y|}}
        +C_1 \left(
        \frac{\inf_{B(y,\frac{1}{4}|y|)}w_0}{\ln\frac{1}{|y|}}
        -\frac{\inf_{B(y,\frac{3}{4}|y|)}w_0}{\ln\frac{1}{|y|}}
        \right)
        +C_2\frac{W_{1,n}^{\mu}(y,\frac{1}{2}|y|)}{\ln\frac{1}{|y|}}.
    \end{equation}
    It follows from \cref{eq. lemma 3.7 of MQ} and \cref{ineq. KM94 upper bound} that, for any $\epsilon>0$, there exists $R_1>1$ such that
    \[
        \frac{w_0(y)}{\ln\frac{1}{|y|}}\leq
        m_1+\frac{\epsilon}{n}+C_2\frac{W_{1,n}^{\mu}(y,\frac{1}{2}|y|)}{\ln\frac{1}{|y|}},
        \quad \forall\, |y|\leq \frac{1}{R_1}.
    \]
    Equivalently,
    \[
        e^{nw_0(y)}\leq 
        \left(\frac{1}{|y|}\right)^{nm_1+\epsilon}
        e^{nC_2W_{1,n}^{\mu}(y,\frac{1}{2}|y|)},
        \quad \forall\, |y|\leq \frac{1}{R_1}.
    \]
Hence if $R\geq R_1$, we have
    \begin{align} \notag
        \int_{B(0,2R)\setminus B(0,R)}e^{nw(x)}\dx
        &=\int_{B(0,\frac{1}{R})\setminus B(0,\frac{1}{2R})}e^{nw_0(y)}\dy \\ \notag
        &\leq \int_{B(0,\frac{1}{R})\setminus B(0,\frac{1}{2R})}\left(\frac{1}{|y|}\right)^{nm_1+\epsilon}
        e^{nC_2W_{1,n}^{\mu}(y,\frac{1}{2}|y|)}\dy \notag\\
        &\leq (2R)^{nm_1+\epsilon}
        \int_{B(0,\frac{1}{R})\setminus B(0,\frac{1}{2R})} e^{nC_2W_{1,n}^{\mu}(y,\frac{1}{2}|y|)}\dy\notag\\
        &\leq C(n)R^{n(2-m)+\epsilon}\int_{B(0,\frac{1}{R})\setminus B(0,\frac{1}{2R})} e^{nC_2W_{1,n}^{\mu}(y,\frac{1}{2}|y|)}\dy. \label{ineq. integral of exp(n phi)}
    \end{align}
    To complete the proof of the lemma, it suffices to show that, for sufficiently large $R$,
    \begin{equation}\label{ineq. integral of Wolff potential}
        \int_{B(0,\frac{1}{R})\setminus B(0,\frac{1}{2R})} e^{nC_2W_{1,n}^{\mu}(y,\frac{1}{2}|y|)}\dy
        \leq C(n) R^{-n}.
    \end{equation}
    
    In fact, if $f\equiv 0$ near the origin, then $\mu(B(y,\frac{1}{2}|y|))=\mu_f(B(y,\frac{1}{2}|y|))=\int_{B(y,\frac{1}{2}|y|)}f\dx=0$ provided $|y|$ is small enough. It follows that $W_{1,n}^{\mu}(y,\frac{1}{2}|y|)\equiv 0$ and \cref{ineq. integral of Wolff potential} holds trivially.
    
    If $f$ is not identically zero in any neighborhood of the origin, then for any $R\geq R_2:=2$, define a nonnegative measure $\mu_R$ as 
    \[
        \mu_R(\Omega):=\frac{\mu_f(\Omega)}{\mu_f(B(0,\frac{2}{R}))},\quad \forall \,\Omega\subset 
        B\left(0,\frac{2}{R}\right).
    \]
    It is clear that $0\leq\mu_R(\Omega)\leq 1$ for $\Omega\subset B(0,\frac{2}{R})$ and $B(y,\frac{1}{2}|y|)\subset B(0,\frac{2}{R})$ for $y\in B(0,\frac{1}{R})\setminus B(0,\frac{1}{2R})$. Moreover, it is straightforward to see that 
    \[
        W_{1,n}^{\mu_R}(y,\frac{1}{2}|y|)
        =\frac{W_{1,n}^{\mu_f}(y,\frac{1}{2}|y|)}{\mu_f(B(0,\frac{2}{R}))^{\frac{1}{n-1}}}
        =\frac{W_{1,n}^{\mu}(y,\frac{1}{2}|y|)}{\mu_f(B(0,\frac{2}{R}))^{\frac{1}{n-1}}}.
    \]

    For brevity, let $\alpha_y^{n-1}:=\mu_R(B(y,\frac{1}{2}|y|))\in [0,1],$ for $y\in B(0,\frac{1}{R})\setminus B(0,\frac{1}{2R})$.  Then 
    \begin{align*}
        W_{1,n}^{\mu_R}(y,\frac{1}{2}|y|)
        &=\int_0^{\frac{1}{2}|y|}\mu_R(B(y,t))^{\frac{1}{n-1}}\dd \ln t \\ \notag
       & =\big(\mu_R(B(y,t))^{\frac{1}{n-1}}\ln t\big)|_0^{\frac{1}{2}|y|}
        +\int_0^{\frac{1}{2}|y|}\ln \frac{1}{t}\dd \mu_R(B(y,t))^{\frac{1}{n-1}}.
    \end{align*}
    Consider the Hardy-Littlewood maximal function of $f$:
    \[
        Mf(y):=
        \sup_{t>0}\frac{\int_{B(y,t)}f\dx}{|B(y,t)|}
        =\mu_f(B(0,\frac{2}{R}))\sup_{t>0}\frac{\mu_R(B(y,t))}{|B(y,t)|}.
    \]
    Hence there holds
    \[
        \mu_R(B(y,t))\leq \frac{|B(0,1)|}{\mu_f(B(0,\frac{2}{R}))}t^n Mf(y)
    \]
    for almost every $y \in B(0,\frac{1}{R})\setminus B(0,\frac{1}{2R})$. Therefore,
    \[
        \big(\mu_R(B(y,t))^{\frac{1}{n-1}}\ln t\big)|_0^{\frac{1}{2}|y|}
        =\alpha_y\ln(\frac{1}{2}|y|).
    \]
    If $\alpha_y>0$, then by Jensen's inequality,
    \begin{align*}
        & e^{W_{1,n}^{\mu_R}(y,\frac{1}{2}|y|)}
        =(\frac{1}{2}|y|)^{\alpha_y} 
        e^{\int_0^{\frac{1}{2}|y|} \ln \frac{1}{t}\dd \mu_R(B(y,t))^{\frac{1}{n-1}}
        }\\
        \leq &(\frac{1}{2}|y|)^{\alpha_y} \int_0^{\frac{1}{2}|y|}\frac{1}{t^{\alpha_y}}\frac{1}{\alpha_y}\dd \mu_R(B(y,t))^{\frac{1}{n-1}}\\
        \leq &(\frac{1}{2}|y|)^{\alpha_y}
        \left(
        \frac{1}{\alpha_y}\frac{1}{t^{\alpha_y}}\mu_R(B(y,t))^{\frac{1}{n-1}}\Big|_0^{\frac{1}{2}|y|}
        +\int_0^{\frac{1}{2}|y|}\mu_R(B(y,t))^{\frac{1}{n-1}}\frac{1}{t^{\alpha_y+1}}\dd t
        \right)\\
        =&(\frac{1}{2}|y|)^{\alpha_y}
        \left(
        \frac{1}{(\frac{1}{2}|y|)^{\alpha_y}}
        +\int_0^{\frac{1}{2}|y|}\mu_R(B(y,t))^{\frac{1}{n-1}}\frac{1}{t^{\alpha_y+1}}\dd t
        \right)\\
        \leq &(\frac{1}{2}|y|)^{\alpha_y}
        \left(
        \frac{1}{(\frac{1}{2}|y|)^{\alpha_y}}
        +\frac{1}{\frac{n}{n-1}-\alpha_y}
        \left(\frac{|B(0,1)|}{\mu_f(B(0,\frac{2}{R}))}\right)^{\frac{1}{n-1}}
        \Big(Mf(y)\Big)^{\frac{1}{n-1}}
        \Big(\frac{1}{2}|y|\Big)^{\frac{n}{n-1}-\alpha_y}
        \right)\\
        =&1
        +\frac{1}{\frac{n}{n-1}-\alpha_y}
        \left(\frac{|B(0,1)|}{\mu_f(B(0,\frac{2}{R}))}\right)^{\frac{1}{n-1}}
        \Big(Mf(y)\Big)^{\frac{1}{n-1}}
        \Big(\frac{1}{2}|y|\Big)^{\frac{n}{n-1}}.
    \end{align*}
    If $\alpha_y=0$, then $W_{1,n}^{\mu_R}(y,\frac{1}{2}|y|)=0$ and there still holds
    \begin{align*}
        e^{W_{1,n}^{\mu_R}(y,\frac{1}{2}|y|)}
        \leq 1
        +\frac{1}{\frac{n}{n-1}-\alpha_y}
        \left(\frac{|B(0,1)|}{\mu_f(B(0,\frac{2}{R}))}\right)^{\frac{1}{n-1}}
        \Big(Mf(y)\Big)^{\frac{1}{n-1}}
        \Big(\frac{1}{2}|y|\Big)^{\frac{n}{n-1}}.
    \end{align*}
    Hence for $\lambda\geq 2$, the weak type Hardy-Littlewood maximal inequality yields
    \begin{align*}
        &\Big|\Big\{
        y\in B(0,\frac{1}{R})\setminus B(0,\frac{1}{2R}):
        e^{W_{1,n}^{\mu_R}(y,\frac{1}{2}|y|)}\geq \lambda
        \Big\}\Big|\\
        \leq& 
        \bigg|\bigg\{
        y\in B(0,\frac{1}{R})\setminus B(0,\frac{1}{2R}):
        Mf(y)\geq \frac{2\mu_f(B(0,\frac{2}{R}))(\frac{n}{n-1}-\alpha_y)^{n-1}}{|B(0,1)|\ |y|^n}\lambda^{n-1}
        \bigg\}\bigg|\\
        \leq&
        \bigg|\bigg\{
        y\in B(0,\frac{1}{R})\setminus B(0,\frac{1}{2R}):
        Mf(y)\geq \frac{2\mu_f(B(0,\frac{2}{R}))(\frac{1}{n-1})^{n-1}}{|B(0,1)|}R^n\lambda^{n-1}
       \bigg\}\bigg|\\
        \leq&
        \frac{C_3(n)|B(0,1)|}{2\mu_f(B(0,\frac{2}{R}))\, (\frac{1}{n-1})^{n-1}R^n\lambda^{n-1}} 
        \int_{B(0,\frac{1}{R})\setminus B(0,\frac{1}{2R})} f\dx\\
        =& \frac{C_4(n)}{R^n}
        \left(\frac{\int_{B(0,\frac{1}{R})\setminus B(0,\frac{1}{2R})} f\dx}{\mu_f(B(0,\frac{2}{R}))}\right)
        \lambda^{-(n-1)}.
    \end{align*}

    Finally, notice that $f\in L^1(B(0,2))$, there exists $R_3>0$ such that $0<\mu_f(B(0,\frac{2}{R_3}))<(\frac{n-1}{nC_2(n)})^{n-1}$, that is to say, $q:=nC_2(n)\mu_f(B(0,\frac{2}{R_3}))^{\frac{1}{n-1}}<n-1$. Now for $R>\max\{R_2,\ R_3\}$, we could verify \cref{ineq. integral of Wolff potential} as follows:
    \begin{align*}
        &\int_{B(0,\frac{1}{R})\setminus B(0,\frac{1}{2R})} e^{nC_2W_{1,n}^{\mu}(y,\frac{1}{2}|y|)}\dy\\
        =&\int_{B(0,\frac{1}{R})\setminus B(0,\frac{1}{2R})} e^{nC_2 \mu_f(B(0,\frac{2}{R}))^{\frac{1}{n-1}} W_{1,n}^{\mu_R}(y,\frac{1}{2}|y|)}\dy\\
        \leq&\int_{B(0,\frac{1}{R})\setminus B(0,\frac{1}{2R})}
        e^{qW_{1,n}^{\mu_R}(y,\frac{1}{2}|y|)}\dy
        \\
        =&\int_0^{+\infty}
        \Big|\Big\{
        y\in B(0,\frac{1}{R})\setminus B(0,\frac{1}{2R}):
        e^{W_{1,n}^{\mu_R}(y,\frac{1}{2}|y|)}
        \geq 
        t^{\frac{1}{q}}
        \Big\}\Big|\dd t\\
        \leq& \int_{2^{q}}^{+\infty}
        \frac{C_4(n)}{R^n}
         \left(\frac{\int_{B(0,\frac{1}{R})\setminus B(0,\frac{1}{2R})} f\dx}{\mu_f(B(0,\frac{2}{R}))}\right)
         t^{\frac{-(n-1)}{q}}
         \dd t
         +\int_0^{2^{q}}
         |B(0,\frac{1}{R})\setminus B(0,\frac{1}{2R})|\dd t\\
         =&C_4(n)
         \left(\frac{\int_{B(0,\frac{1}{R})\setminus B(0,\frac{1}{2R})} f\dx}{\mu_f(B(0,\frac{2}{R}))}\right)
         \frac{q}{n-1-q}
         \left(\frac{1}{2}\right)^{n-1-q}
         R^{-n}
         +2^{q}
         |B(0,\frac{1}{R})\setminus B(0,\frac{1}{2R})|\\
         \leq& C_5(n)R^{-n}.
    \end{align*}
    Combining with \cref{ineq. integral of exp(n phi)} and \cref{ineq. integral of Wolff potential}, we conclude that $R_0$ could be chosen as $\max\{R_1,R_2,R_3\}+1$ and the proof is completed.
\end{proof}

Now we are ready to prove \cref{proposition. third part}.
\begin{proof}[Proof of \cref{proposition. third part}:]
    Fix $\epsilon>0$, by \cref{lemma. integral of exp(n phi)}, there exists $R_0>0$ such that, for any $R\geq R_0$ and $1\leq i\leq k$, we have
    \[
        \int_{B(0,2^iR)\setminus B(0,2^{i-1}R) }e^{nw(x)}\dx
        \leq C(n)2^{(i-k-1)(n(1-m)+\epsilon)}(2^k R)^{n(1-m)+\epsilon}.
    \]
    Summing from $i=1$ to $k$, we get
    \begin{equation}\label{ineq.integral over 2 to the k}
        \int_{B(0,2^k(R))\setminus B(0,R)}e^{nw(x)}\dx
        \leq C(n)(2^k R)^{n(1-m)+\epsilon}.
    \end{equation}
    Now for any $\tilde{R}\geq 2R_0$, there exists some integer $k\geq 1$ such that $2^{-k}\tilde{R}\in[R_0, 2R_0)$. Therefore, \cref{ineq.integral over 2 to the k} implies that
    \begin{align}
        \int_{B(0,\tilde{R})}e^{nw(x)}\dx
        &=\int_{B(0,\tilde{R})\setminus B(0,2^{-k}\tilde{R}) }
        e^{nw(x)}\dx
        +\int_{B(0,2^{-k}\tilde{R})}
        e^{nw(x)}\dx\notag\\
        &\leq C(n) \tilde{R}^{n(1-m)+\epsilon}
        +\int_{B(0,2R_0)}
        e^{nw(x)}\dx\label{ineq. upper bound of integral of e to the n phi}
    \end{align}

    Define $\bar{\rho}(r):=\sup_{\partial B^g(0,r) }|x|=e^{s(r)}$, where $s(r):=\sup_{\partial B^g(0,r)}\ln |x| $. It is easy to see that 
    $B(0,\bar{\rho}(r))\supset B^g(o,r)$.
    Then it follows that
    \begin{equation}\label{eq.supset volume relation}
        \operatorname{vol}_g(B(0,\bar{\rho}(r))
        \geq \operatorname{vol}_g(B^g(0,r)).
    \end{equation}

    On the one hand, \cref{ineq. upper bound of integral of e to the n phi} yields that for $\bar{\rho}(r)\geq 2R_0$,
    \[
        \operatorname{vol}_g(B(0,\bar{\rho}(r))
        =\int_{B(0,\bar{\rho}(r))} e^{nw(x)}\dx
        \leq
        C(n)\bar{\rho}(r)^{n(1-m)+\epsilon}
        +\int_{B(0,2R_0)}
        e^{nw(x)}\dx. 
    \]
    
    On the other hand, by our assumption that $\beta>0$, there exists a constant $C=C(\beta,n)>0$ such that
    \[
        \operatorname{vol}_g(B^g(0,r))\geq C(\beta,n)r^n,\quad \forall r>0.
    \]
    Inserting these into \cref{eq.supset volume relation}, then we derive that
    \[
        C\bar{\rho}(r)^{n(1-m)+\epsilon}
        +\int_{B(0,2R_0)}
        e^{nw(x)}\dx
        \geq C r^n.
    \]
    Therefore,
    \begin{align*}
        \frac{n\ln r+C}{s(r)}&\leq \frac{\ln(C\bar{\rho}(r)^{n(1-m)+\epsilon}+\int_{B(0,2R_0)}
        e^{nw(x)}\dx)}{s(r)} \\ \notag
        &=\frac{\ln (e^{(n(1-m)+\epsilon)s(r) }+C\int_{B(0,2R_0)}
        e^{nw(x)}\dx)+C}{s(r)}
    \end{align*}
    After taking limit superior, we obtain
    \[
        \limsup_{r\to \infty}\frac{\ln r}{s(r)}
        \leq 1-m+\frac{\epsilon}{n}.
    \]
    Since $\epsilon$ is arbitrary, we derive that
    \[
        \limsup_{r\to \infty}\frac{\ln r}{s(r)}
        \leq 1-m.
    \]

    Finally we apply \cref{proposition. Harnack ineq.} to complete the proof as follows:
    \[
        \limsup_{x\to\infty}\frac{\ln r(x)}{\ln |x|}\leq 
        \limsup_{r\to \infty}\frac{\ln r}{i(r)}
        \leq \limsup_{r\to \infty}\frac{e^{o(1)}\ln r}{s(r)}
        \leq 1-m.
    \]
\end{proof}

\subsection{Upper bound by $\beta$ }
\ \\
In this subsection, we also include a result showing that the ratio of two distances is bounded from above by the asymptotic volume ratio $\beta$. This is similar to that of \cite[Theorem 2.5]{LT2}.

\begin{proposition}\label{proposition. first part}
    Let $(\mathbb{R}^n,e^{2w}g_0) (n\geq 3)$ be a complete Riemannian manifold with nonnegative Ricci curvature and  $\beta>0$.
    Then 
    \[
        \limsup_{x\to \infty}\frac{\ln r(x)}{\ln |x|}\leq \beta^{\frac{1}{n-1}}.
    \]
\end{proposition}

 Define the area function  
$A(r):=\operatorname{Area}_g(\partial B^g(0,r)).$

We begin with a lower bound of the integral of some power of $A(r)$.
\begin{lemma}\label{lemma. lower bound of integral of A(r)}
    Let $(\mathbb{R}^n,e^{2w}g_0)$ be a complete Riemannian manifold with nonnegative Ricci curvature and $\beta>0$. Then for any $\epsilon>0$, there exists $R_0>0$ such that
    \[
        |\mathbb{S}^{n-1}|^{\frac{1}{n-1}} 
        \int_{R_1}^{R_2} A^{-\frac{1}{n-1}}(t)\dd t
        \geq (\beta+\epsilon)^{-\frac{1}{n-1}}
        (\ln R_2-\ln R_1),\quad \forall R_2>R_1\geq R_0.
    \]
\end{lemma}
\begin{proof}

It suffices first to work with regular radii $R_1$ and $R_2$,
for which the corresponding distance spheres are smooth up to a
set of measure zero. The general statement then follows by
approximation and the coarea formula.

    By the definition of $\beta$ and the Bishop-Gromov volume comparison theorem, for any $\epsilon>0$, there exists $R_0>0$ such that
    \[
        \beta+\epsilon\geq \frac{A(r)}{|\mathbb{S}^{n-1}|r^{n-1}},\qquad \forall r\geq R_0.
    \]
    Equivalently,
    \[
        |\mathbb{S}^{n-1}|^{\frac{1}{n-1}} 
        A(r)^{-\frac{1}{n-1}}
        \geq (\beta+\epsilon)^{-\frac{1}{n-1}}\frac{1}{r},\qquad \forall r\geq R_0.
    \]
    Integrating this inequality from $R_1$ to $R_2$, we get the desired inequality.
    \end{proof}
    
Now we turn to an upper bound of $\int_{R_1}^{R_2} A^{-\frac{1}{n-1}}(t)\dd t$.
\begin{lemma}\label{lemma. upper bound of integral of A(r)}
    Let $(\mathbb{R}^n,e^{2w}g_0),(n\geq 3)$ be a complete Riemannian manifold. Denote $G(x):=\ln |x|$ and $s(r):=\sup_{\partial B^g(0,r)}G,\, i(r):=\inf_{\partial B^g(0,r)}G $. Then
    \[
        s(R_2)-i(R_1)
        \geq        
        |\mathbb{S}^{n-1}|^{\frac{1}{n-1}}
        \left(
        \int_{R_1}^{R_2} A^{-\frac{1}{n-1}}(t)\dd t
        \right),\qquad 
        \forall\, 0<R_1<R_2.
    \]
\end{lemma}
\begin{proof}
    Let $h$ solve the equation 
    \[
        \begin{cases}
            \Delta_n^g h=0  &\text{in } B^g(0,R_2)\setminus B^g(0,R_1),\\
            h=i(R_1) &\text{on } \partial B^g(0,R_1),\\
            h=s(R_2) &\text{on } \partial B^g(0,R_2).
        \end{cases}
    \]
    It follows that
    \begin{align}
        \int_{B^g(0,R_2)\setminus B^g(0,R_1)}&|\nabla^g h|_g^n dV_g\notag
        =-\int_{B^g(0,R_2)\setminus B^g(0,R_1)}h\Delta_n^g h dV_g \\ \notag
        &+\int_{\partial B^g(0,R_2)} h|\nablag h|_g^{n-2}\frac{\partial h}{\partial r} dA_g
        -\int_{\partial B^g(0,R_1)} h|\nablag h|_g^{n-2}\frac{\partial h}{\partial r} dA_g \notag\\
        =&s(R_2)\int_{\partial B^g(0,R_2)} |\nablag h|_g^{n-2}\frac{\partial h}{\partial r} dA_g
        -i(R_1)\int_{\partial B^g(0,R_1)} |\nablag h|_g^{n-2}\frac{\partial h}{\partial r} dA_g \notag\\
        =&(s(R_2)-i(R_1))\int_{\partial B^g(0,R_2)} |\nablag h|_g^{n-2}\frac{\partial h}{\partial r} dA_g. \label{eq. n-Dirichlet energy of h}
    \end{align}

    Let $f$ satisfy
    \[
        \begin{cases}
            \Delta_n^g f=0 &\text{in } B^g(0,R_2)\setminus B^g(0,R_1),\\
            f=G &\text{on } \partial B^g(0,R_1),\\
            f=s(R_2) &\text{on } \partial B^g(0,R_2).
        \end{cases}
    \]
Applying the comparison principle for $n$-harmonic functions (see \cite[Theorem 2.15]{L}) to $f$ and $h$, we get
    \[
        f\geq h \quad \text{on } B^g(0,R_2)\setminus B^g(0,R_1).
    \]
    It follows from $f=h$ on $\partial B^g(0,R_2)$ that 
    \[
        \frac{\partial f}{\partial r} \leq \frac{\partial h}{\partial r} \quad \text{on } \partial B^g(0,R_2).
    \]
    We claim that
    \begin{equation}\label{ineq. gradient comparison}
        |\nablag f|_g^{n-2}\frac{\partial f}{\partial r}
        \leq  
        |\nablag h|_g^{n-2}\frac{\partial h}{\partial r} 
        \quad \text{on } \partial B^g(0,R_2).
    \end{equation}
    In fact, there holds $\nabla_{\partial B^g(0,R_2)}^g f=\nabla_{\partial B^g(0,R_2)}^g h$ since $f=h$ on $\partial B^g(0,R_2)$. Then We shall verify \cref{ineq. gradient comparison} case by case using a trick in \cite[Theorem 5.2]{MQ}.
    \begin{itemize}
        \item If $0\leq \frac{\partial f}{\partial r}$, then 
        \[
            |\nablag f|_g^2=|\nabla_{\partial B^g(0,R_2)}^g f|_g^2+(\frac{\partial f}{\partial r})^2
            \leq |\nablag h|_g^2
             \quad \text{on } \partial B^g(0,R_2).
        \]
        Hence \cref{ineq. gradient comparison} holds.
        \item If $\frac{\partial f}{\partial r}<0\leq \frac{\partial h}{\partial r}$, then
        \[
            |\nablag f|_g^{n-2}\frac{\partial f}{\partial r}<0\leq |\nablag h|_g^{n-2}\frac{\partial h}{\partial r}
             \quad \text{on } \partial B^g(0,R_2).
        \]
        \item If $\frac{\partial f}{\partial r}\leq \frac{\partial h}{\partial r}<0$, then
        \[
            |\nablag f|_g^2=|\nabla_{\partial B^g(0,R_2)}^g f|_g^2+(\frac{\partial f}{\partial r})^2
            \geq |\nablag h|_g^2
             \quad \text{on } \partial B^g(0,R_2).
        \]
        Hence \cref{ineq. gradient comparison} still holds in this case.
    \end{itemize}

    Similarly, we apply the comparison principle for $n$-harmonic functions to $f$ and $G$ to derive that
    \begin{align}
        f\geq G \quad\quad\text{on }& B^g(0,R_2)\setminus B^g(0,R_1)\notag\\
        |\nablag f|_g^{n-2}\frac{\partial f}{\partial r}
        \geq  
        |\nablag G|_g^{n-2}\frac{\partial G}{\partial r} 
        \quad\text{on }& \partial B^g(0,R_1).\label{ineq. gradient comparison between f and G}
    \end{align}

    Therefore, by \cref{ineq. gradient comparison} and \cref{ineq. gradient comparison between f and G}, we could get a lower bound of \cref{eq. n-Dirichlet energy of h}:
    \begin{align}
        \int_{B^g(0,R_2)\setminus B^g(0,R_1)}|\nabla^g h|_g^n\dd V_g
        =&(s(R_2)-i(R_1))\int_{\partial B^g(0,R_2)} |\nablag h|_g^{n-2}\frac{\partial h}{\partial r}\dd A_g \notag\\
        \geq &(s(R_2)-i(R_1))\int_{\partial B^g(0,R_2)} |\nablag f|_g^{n-2}\frac{\partial f}{\partial r}\dd A_g\notag\\
        =&(s(R_2)-i(R_1))\int_{\partial B^g(0,R_1)} |\nablag f|_g^{n-2}\frac{\partial f}{\partial r}\dd A_g\notag\\
        \geq & (s(R_2)-i(R_1))\int_{\partial B^g(0,R_1)} |\nablag G|_g^{n-2}\frac{\partial G}{\partial r}
        \dd A_g.
        \label{ineq. lower bound of n-Dirichlet energy of h}
    \end{align}

    On the other hand, consider 
    \[
        H(x):=\frac{(s(R_2)-i(R_1))\int_{R_1}^{r(x)}A^{-\frac{1}{n-1}}(t)\dd t}{\int_{R_1}^{R_2} A^{-\frac{1}{n-1}}(t)\dd t}
        +i(R_1),\quad x\in B^g(0,R_2)\setminus B^g(0,R_1).
    \]
    It is straightforward to see that $H|_{\partial B^g(0,R_1)}=i(R_1),\, H|_{\partial B^g(0,R_2)}=s(R_2)$ and
    \[
        \nablag H=\frac{s(R_2)-i(R_1)}{\int_{R_1}^{R_2} A^{-\frac{1}{n-1}}(t)\dd t}
        A^{-\frac{1}{n-1}}(r)\nabla^gr,\quad x\in B^g(0,R_2)\setminus B^g(0,R_1).
    \]
    Therefore we could derive an upper bound of \cref{eq. n-Dirichlet energy of h}:
    \begin{align}
        \int_{B^g(0,R_2)\setminus B^g(0,R_1)}|\nabla^g h|_g^n\dd V_g\notag
        \leq& \int_{B^g(0,R_2)\setminus B^g(0,R_1)}|\nabla^g H|_g^n\dd V_g\notag\\
        =&\int_{R_1}^{R_2}\left(
        \int_{\partial B^g(0,r)}\frac{(s(R_2)-i(R_1))^n}{(\int_{R_1}^{R_2} A^{-\frac{1}{n-1}}(t)\dd t)^n} A^{-\frac{n}{n-1}}(r)\dd A_g
        \right)\dd r\notag\\
        =&\int_{R_1}^{R_2}
        \frac{(s(R_2)-i(R_1))^n}{(\int_{R_1}^{R_2} A^{-\frac{1}{n-1}}(t)\dd t)^n} A^{-\frac{1}{n-1}}(r)
        \dd r\notag\\
        =&(s(R_2)-i(R_1))^n
        \left(
        \int_{R_1}^{R_2} A^{-\frac{1}{n-1}}(t)\dd t
        \right)^{-(n-1)}.\label{ineq. upper bound of n-Dirichlet energy of h}
    \end{align}
    Combining  \cref{ineq. lower bound of n-Dirichlet energy of h} and \cref{ineq. upper bound of n-Dirichlet energy of h}, we derive, for any $0<R_1<R_2$,
    \[
        \left(\int_{\partial B^g(0,R_1)}|\nablag G|_g^{n-2}\frac{\partial G}{\partial r}\dd A_g\right)
        \left(
        \int_{R_1}^{R_2} A^{-\frac{1}{n-1}}(t)\dd t
        \right)^{n-1}
        \leq 
        \big(
        s(R_2)-i(R_1)
        \big)^{n-1}.
    \]

    Finally, recall that 
    \[
        \Delta_n G=\Delta_n \ln |x|=|\mathbb{S}^{n-1}|\delta_0,\quad \text{on }\mathbb{R}^n
    \]
    in the distribution sense.  Hence we could complete the proof by
    \[
        \int_{\partial B^g(0,R_1)}|\nablag G|_g^{n-2}\frac{\partial G}{\partial r}\dd A_g
        =\int_{B^g(0,R_1)}\Delta_n^g
        \ln |x| \ d V_g
        =\int_{B^g(0,R_1)}\Delta_n\ln |x| \dx
        =|\mathbb{S}^{n-1}|.
    \]
\end{proof}

\begin{proof}[Proof of \cref{proposition. first part}:]
    By \cref{lemma. lower bound of integral of A(r)} and \cref{lemma. upper bound of integral of A(r)}, for any $\epsilon>0$, there exists $R_0>0$ such that $\forall R_2>R_1\geq R_0$
    \[
        (\beta+\epsilon)^{-\frac{1}{n-1}}
        (\ln R_2-\ln R_1)
        \leq 
        |\mathbb{S}^{n-1}|^{\frac{1}{n-1}}
        \left(
        \int_{R_1}^{R_2} A^{-\frac{1}{n-1}}(t)\dd t
        \right)
        \leq s(R_2)-i(R_1).
    \]
    Then \cref{proposition. Harnack ineq.} implies that
    \[
        (\beta+\epsilon)^{-\frac{1}{n-1}}(\ln R_2-\ln R_1)
        +i(R_1)
        \leq s(R_2)
        \leq i(R_2)e^{o(1)},\quad  \text{as } R_2\to +\infty.
    \]
    It follows that
    \[
        \limsup_{|x|\to+\infty}\frac{\ln r(x)}{\ln |x|}
        \leq \limsup_{R_2\to +\infty}\frac{\ln R_2}{i(R_2)}\leq (\beta+\epsilon)^{\frac{1}{n-1}}.
    \]
The proof is finished by sending $\epsilon$ to zero.
\end{proof}

\section{Weighted $G$-convergence}
\label{sec:weighted-G}

In this section we prove the convergence statement (\cref{thm:weighted-G-compactness}) needed in the subsequent blow-down argument .  The terminology ``weighted $G$-convergence'' is used only to emphasize the analogy with the classical $G$-compactness theorem for uniformly elliptic divergence-form operators
 (see, for example, \cite[Theorem 3]{ZKON}). 

Let
$g=e^{2w}g_0$
be a complete conformal metric on $\mathbb R^n$, $n\geq 3$, with
$\operatorname{Ric}_g\geq 0$ and $\beta>0$.  We write
\[
    \underline{w}(R):=\inf_{|x|=R}w(x).
\]
Let $m\in(0,1)$ be the asymptotic exponent of $w$ (see  \cref{def of m}), and set
$a:=1-m.$ 

Let $h$ be a $g$-harmonic function, then the conformal transformation formula gives
\[
    \Delta_g h=0
    \quad\Longleftrightarrow\quad
    \operatorname{div}\bigl(e^{(n-2)w}\nabla h\bigr)=0,
\]
where $\operatorname{div},\nabla$ are the divergence and the gradient in the standard Euclidean space respectively. 

In view of the asymptotic exponent of the conformal factor \cref{eq:asympradial}, one may expect that a tangent cone at infinity of $(\mathbb{R}^n, g)$ is the metric cone $(\mathbb{R}^n, |x|^{-2m} g_0)$. It is probably difficult to fully establish this geometric statement due to the presence of that strong 
$\mathcal{E}$-set in \cref{eq:asympradial}.
Nevertheless, one can blow down a $g$-harmonic function to a harmonic function on the limit cone, as we now explain.

For $R>0$, define
\begin{equation}\label{eq:def-AR-BR}
    P_R(y):=\exp\bigl((n-2)(w(Ry)-\underline w(R))\bigr),
    \qquad
    Q_R(y):=\exp\bigl(n(w(Ry)-\underline w(R))\bigr).
\end{equation}
The limiting weights are
\[
    P_\infty(y):=|y|^{-m(n-2)},
    \qquad
    Q_\infty(y):=|y|^{-mn}.
\]

Set $u_R(y):=c_R^{-1}h(Ry), (c_R\neq 0)$, then it solves
\[
    \operatorname{div}(P_R\nabla u_R)=0.
\]

We shall prove that a proper blow down limit of a $g$-harmonic function $h$ converges to a solution of the limiting equation
\begin{equation}\label{eq:limiting-operator}
    L_\infty u
    :=\operatorname{div}(P_\infty\nabla u)
    =\operatorname{div}\bigl(|y|^{-m(n-2)}\nabla u\bigr),
\end{equation}
 which is exactly the
Laplace--Beltrami operator of the limiting cone
\begin{equation}\label{eq. def of limiting cone}
    \mathcal{C}_a=
    \bigl([0,\infty)\times\mathbb S^{n-1},
    d\rho^2+a^2\rho^2g_{\mathbb S^{n-1}}\bigr),
    \qquad \rho=\frac{1}{a}|y|^a.
\end{equation}
{More explicitly, we start with a convergence result for the operators $\operatorname{div}(P_R\nabla\cdot)$.
Ideally, one would hope for the sequence of operators  $\operatorname{div}(P_{R}\nabla\cdot)$ to converge to the limiting operator $\operatorname{div}(P_{\infty}\nabla\cdot)$ in the well-studied $G$-convergence framework, since then Theorem \ref{thm:T2} would follow directly from \cite[Theorem 2.13]{L}.
 However,  the presence of the strong $\mathcal{E}$-set in the asymptotic behavior of $w$ and the lack of boundedness of the limiting weight $P_{\infty}$ make such $G$-convergence unattainable.
 To overcome this obstacle, we instead show that the operators are $\operatorname{III}$-convergent 
in the sense of \cite[Page 76]{ZKON}. 
Building on this, we then establish a matched compactness result for solutions, which plays a vital role in subsequent blow-down arguments.}

{First, we isolate several precise consequences of the asymptotic analysis of the conformal factor $w$, resulting the $\operatorname{III}$-convergence of the operators $\operatorname{div}(P_R\nabla\cdot)$ (i.e. \cref{eq:AR-L1-ball}).}
The proof utilizes the convergence of
$r\underline w'(r)$, the spherical exponential-average estimate, and the
annular $g$-volume estimate established in \cite{M}.
\begin{proposition}
\label{prop:weight-convergence}
For every $\Lambda>0$, as $R\to\infty$,
\begin{equation}\label{eq:AR-L1-ball}
    P_R\to P_\infty
    \quad\text{in }L^1(B_\Lambda),
\end{equation}
\begin{equation}\label{eq:BR-L1-ball}
    Q_R\to Q_\infty
    \quad\text{in }L^1(B_\Lambda).
\end{equation}
In particular, 
\[
    \sup_{R\geq R_\Lambda}\int_{B_\Lambda}P_R\dd y<\infty,
    \qquad
    \sup_{R\geq R_\Lambda}\int_{B_\Lambda}Q_R\dd y<\infty.
\]
Moreover, there exists $c_\Lambda>0$ and $R_\Lambda>0$ such that
\begin{equation}\label{eq:conformal-lower-bound}
    e^{w(Ry)-\underline w(R)}\geq c_\Lambda
    \qquad\text{for all }y\in B_\Lambda,
    \quad R\geq R_\Lambda.
\end{equation}
\end{proposition}

\begin{proof}
The radial analysis in \cite[Lemma 2.6]{M} gives
\begin{equation}\label{eq:radial-derivative-limit}
    r\underline w'(r)\to -m \qquad
    \text{as } r\to +\infty.
\end{equation}
Hence, for every $0<\delta<\Lambda<\infty$,
\[
    \underline w(Rs)-\underline w(R)
    \to -m\log s,\qquad
    \text{as }R\to +\infty,
\]
uniformly for $s\in[\delta,\Lambda]$.  Therefore, for $q=n-2$ and
$q=n$,
\begin{equation}\label{eq:radial-weight-uniform}
    \exp\bigl(q(\underline w(R|y|)-\underline w(R))\bigr)
    \to |y|^{-mq},\qquad
    \text{as }R\to +\infty,
\end{equation}
uniformly on every compact annulus.

On a noncompact region,  a direct  integration of \cref{eq:radial-derivative-limit}  implies the following estimate: for any $\epsilon>0$, there exists $R_\epsilon>0$ such that for all $R\geq R_\epsilon$ and $s\geq \frac{R_\epsilon}{R}$, there holds
\begin{equation}\label{eq:potter-bound}
     \exp\bigl(q(\underline w(Rs)-\underline w(R))\bigr)
    \leq \max\{ s^{-q(m-\epsilon)}, s^{-q(m+\epsilon)} \}, 
    \qquad q\in\{n-2,n\}.
\end{equation}

We first prove the convergence of $P_R$.  Set
\[
    \underline P_R(y)
    :=\exp\bigl((n-2)(\underline w(R|y|)-\underline w(R))\bigr).
\]
Since $w(x)\geq\underline w(|x|)$, there holds $P_R\geq \underline P_R.$
For a fixed annulus $A_{\delta,\Lambda}$, we have
\begin{equation}\label{eq:A-annulus-difference}
\int_{A_{\delta,\Lambda}}(P_R-\underline P_R)\dd y
=
|\mathbb S^{n-1}|
\int_\delta^\Lambda s^{n-1}
\exp\bigl((n-2)(\underline w(Rs)-\underline w(R))\bigr)
\bigl(F(Rs)-1\bigr)\dd s,
\end{equation}
where
\[
    F(t):=
    \frac{1}{|\mathbb S^{n-1}|}\int_{\mathbb S^{n-1}}
    \exp\bigl((n-2)(w(t\theta)-\underline w(t))\bigr)\dd \theta.
\]
By the spherical exponential-average estimate in
\cite[Lemma 4.7]{M}, there holds
\begin{equation}\label{eq:spherical-average-limit}
    F(t)\to 1\qquad 
    \text{as }t\to+\infty.
\end{equation}
 It follows from
\cref{eq:A-annulus-difference}, \cref{eq:radial-weight-uniform}, and
\cref{eq:spherical-average-limit} that
\[
    \|P_R-\underline P_R\|_{L^1(A_{\delta,\Lambda})}
    \to 0\qquad
    \text{as } R\to+\infty.
\]
Together with \cref{eq:radial-weight-uniform}, this proves
\begin{equation}\label{eq:A-L1-annulus}
    P_R\to P_\infty
    \quad\text{in }L^1(A_{\delta,\Lambda}).
\end{equation}

To pass from annuli to the whole ball, we prove uniform integrability near 0. Fix $T>1$. On the central ball $B_{T / R}$, changing variables $x=R y$ yields
\[
\int_{B_{T / R}} P_R(y) \mathrm{d} y=R^{-n} e^{-(n-2) \underline{w}(R)} \int_{B_T} e^{(n-2) w(x)} \mathrm{d} x.
\]

Since $\underline{w}(R)=-m \log R+o(\log R)$ and $n-m(n-2)>0$, the right-hand side tends to zero for each fixed $T$. On $B_\delta \backslash B_{T / R}$, decomposing into dyadic annuli and use \cref{eq:potter-bound} together with \cref{eq:A-annulus-difference} and \cref{eq:spherical-average-limit}, we obtain
\begin{equation} \label{eq: 4.1}
\lim _{\delta \downarrow 0} \limsup _{R \rightarrow \infty} \int_{B_\delta} P_R \mathrm{~d} y=0 
\end{equation}
Since the same estimate is immediate for $P_{\infty}$, combining \cref{eq:A-L1-annulus} and \cref{eq: 4.1} proves \cref{eq:AR-L1-ball}.

We now turn to $Q_R$.  Define
\[
    \underline Q_R(y)
    :=\exp\bigl(n(\underline w(R|y|)-\underline w(R))\bigr).
\]
Again $Q_R\geq\underline Q_R$, and $\underline Q_R\to Q_\infty$
uniformly on compact annulus. In the proof of \cite[Page 25]{M}, one obtains, uniformly for
$\tau\in[1/2,1]$,
\begin{equation}\label{eq:Ma-annular-volume}
\frac{
\displaystyle
\int_{A_{\tau R/2,\tau R}}
\bigl(e^{nw(x)}-e^{n\underline w(|x|)}\bigr)\dd x
}{
\displaystyle
\int_{A_{\tau R/2,\tau R}}
 e^{n\underline w(|x|)}\dd x
}
\to 0, \qquad \text{as $R\to\infty $}.
\end{equation}
Here the numerator is nonnegative because
$w(x)\geq\underline w(|x|)$.  After the change of variables $x=Ry$,
\cref{eq:Ma-annular-volume} implies
\[
    \|Q_R-\underline Q_R\|_{L^1(A_{\delta,\Lambda})}
    \to0
\]
for every fixed annulus. Thus
\[
    Q_R\to Q_\infty
    \quad\text{in }L^1(A_{\delta,\Lambda}).
\]
The same central-ball and dyadic-annulus argument as above, now using
$mn<n$, \cref{eq:potter-bound} and \cref{eq:Ma-annular-volume}, gives
\[
    \lim_{\delta\downarrow0}
    \limsup_{R\to\infty}
    \int_{B_\delta}Q_R\dd y=0.
\]
This proves \cref{eq:BR-L1-ball}.

It remains to prove the lower bound \cref{eq:conformal-lower-bound}. We shall verify it in three cases separately.
If $0<|y|\leq 1$ and $R|y|$ sufficiently large, then $\underline{w}$ is decreasing due to \cref{eq:radial-derivative-limit}. Therefore, 
\[
   w(Ry)-\underline{w}(R)
   \geq \underline w(R|y|)-\underline w(R)
   \geq0.
\]
If $0<|y|\leq 1$ and $R|y|$ remains in a fixed compact interval, then $w(Ry)$ is bounded
from below while $\underline w(R)\to-\infty$ as $R$ tends to $+\infty$, so the same expression is bounded from below.
If $1\leq |y|\leq\Lambda$,
\cref{eq:radial-derivative-limit} gives
\[
    w(Ry)-\underline{w}(R)\geq
    \underline w(R|y|)-\underline w(R)
    \geq-(m+1)\log\Lambda
\]
for all sufficiently large $R$.  This proves
\cref{eq:conformal-lower-bound}.
\end{proof}

{In the remaining part of this section, we prove the key convergence theorem for solutions to  $\operatorname{div}(A_R\nabla u)=0$.
This can be viewed as an analogue, within the framework of $\operatorname{III}$-convergence of operators (as defined in \cite[Page 76]{ZKON}), of the classical convergence result for $G$-convergent operators (see, for example, \cite[(iv) in page 220]{L}).}

We denote
\[
    d\nu_R:=Q_R(y)\dd y,
    \qquad
    d\nu_\infty:=Q_\infty(y)\dd y,
\]
and, for $r>0$, define the weighted integral averages
\[
    \operatorname{Av}_{R,r}(f)
    :=\frac{\displaystyle\int_{B_r}f\dd\nu_R}
            {\displaystyle\nu_R(B_r)},
    \qquad
    \operatorname{Av}_{\infty,r}(f)
    :=\frac{\displaystyle\int_{B_r}f\dd\nu_\infty}
            {\displaystyle\nu_\infty(B_r)}.
\]

\begin{theorem}
\label{thm:weighted-G-compactness}
Let $R_j\to\infty$ and $\Lambda\geq 1$ are fixed. Assume $u_j$ are weak solutions of
\begin{equation}\label{eq:uj-equation}
    \operatorname{div}(P_{R_j}\nabla u_j)=0
    \qquad\text{in }B_\Lambda,
\end{equation}
with
\begin{equation}\label{eq:weighted-L2-assumption}
    \sup_j\operatorname{Av}_{R_j,\Lambda}(u_j^2)<\infty.
\end{equation}
Then, after passing to a subsequence, there exists a function $u_\infty$
satisfying:
\begin{enumerate}
\item For every $0<\Lambda'<\Lambda$,
\[
    u_j\rightharpoonup u_\infty
    \quad\text{in }H^1(B_{\Lambda'}),\quad 
    \text{and}\quad
    u_j\to u_\infty
    \quad\text{in }L^2(B_{\Lambda'}).
\]

\item  $u_\infty$ is locally bounded near the
origin and satisfies
\[
    L_\infty u_\infty=0
    \qquad\text{in }B_\Lambda.
\]
in the weak sense, where $L_{\infty}$ is defined in \cref{eq:limiting-operator}. 
\item For every $0<r<\Lambda$,
\begin{equation}\label{eq:weighted-mass-convergence}
    \operatorname{Av}_{R_j,r}(u_j^2)
    \to
    \operatorname{Av}_{\infty,r}(u_\infty^2).
\end{equation}
More generally, if
$u_{1,j},\ldots,u_{q,j}$ are finitely many sequences satisfying the
same assumptions and converging simultaneously, then
\begin{equation}\label{eq:cross-mass-convergence}
    \operatorname{Av}_{R_j,r}
    \bigl(u_{p,j}u_{q,j}\bigr)
    \to
    \operatorname{Av}_{\infty,r}
    \bigl(u_{p,\infty}u_{q,\infty}\bigr)
\end{equation}
for all $p,q$ and every fixed $r\in(0,\Lambda)$.


\end{enumerate}
\end{theorem}

\begin{proof}
{\bf Step 1: }
Set
\[
    \widehat g_j
    :=e^{2(w(R_jy)-\underline w(R_j))}g_0.
\]
Then $\widehat g_j
    =(R_je^{\underline w(R_j)})^{-2}D_{R_j}^*g$, where $D_{R_j}(y)=R_jy$.
Consequently, $(\mathbb R^n,\widehat g_j)$ is a complete manifold with nonnegative Ricci curvature,
and its asymptotic volume ratio is equal to the asymptotic volume ratio
$\beta>0$ of $(\mathbb R^n,g)$.  Moreover,
\[
    \dd \mu_{\widehat g_j}=Q_{R_j}\dd y=\dd \nu_{R_j},
\]
and \cref{eq:uj-equation} is equivalent to
\[
    \Delta_{\widehat g_j}u_j=0.
\]

Fix $0<\Lambda'<\Lambda_0<\Lambda$.  By
\cref{eq:conformal-lower-bound}, there exists $c_\Lambda>0$ such
that
\[
    \widehat g_j\geq c_\Lambda^2g_0
    \qquad\text{on }B_\Lambda
\]
for all sufficiently large $j$.  Hence one can choose
$\rho_0=\rho_0(\Lambda-\Lambda_0,c_\Lambda)>0$ such that
\[
    B_{\widehat g_j}(p,2\rho_0)\subset B_\Lambda
    \qquad\text{for every }p\in B_{\Lambda_0}.
\]
Since $u_j^2$ is $\widehat g_j$-subharmonic, the mean-value inequality on
manifolds with nonnegative Ricci curvature \cite[Theorem 2.1]{LS} gives
\begin{equation}\label{eq:mean-value}
    |u_j(p)|^2
    \leq
    \frac{C(n)}{
    \operatorname{vol}_{\widehat g_j}
    (B_{\widehat g_j}(p,\rho_0))}
    \int_{B_{\widehat g_j}(p,\rho_0)}u_j^2\dd\nu_{R_j}
    \leq \frac{C(n)\int_{B_\Lambda}
    u_j^2\dd\nu_{R_j}}{
    \operatorname{vol}_{\widehat g_j}
    (B_{\widehat g_j}(p,\rho_0))}
    ,\quad 
    \forall p\in B_{\Lambda_0}.
\end{equation}

 By
\cref{eq:weighted-L2-assumption} and
\cref{prop:weight-convergence}, we have
\[
    \sup_j\int_{B_\Lambda}u_j^2\dd\nu_{R_j}<\infty.
\]
By Bishop--Gromov volume comparison and the positivity of the asymptotic volume ratio,
\[
    \operatorname{Vol}_{\widehat g_j}
    (B_{\widehat g_j}(p,\rho_0))
    \geq\beta|\mathbb{B}^n|\rho_0^n.
\]
Inserting these estimates into \cref{eq:mean-value}, we get a uniform upper bound of $L^\infty$ norm
\begin{equation}\label{eq:uniform-Linf}
    \sup_j\|u_j\|_{L^\infty(B_{\Lambda_0})}<\infty.
\end{equation}

{\bf Step 2:} 
Choose $\eta\in C_c^\infty(B_{\Lambda_0})$ with
$\eta\equiv1$ on $B_{\Lambda'}$.  Testing
\cref{eq:uj-equation} with $\eta^2u_j$ gives
\[
    \int_{B_{\Lambda_0}} P_{R_j}\eta^2|\nabla u_j|^2
    \leq 4\int_{B_{\Lambda_0}}
    P_{R_j}u_j^2|\nabla\eta|^2.
\]
The right-hand side is uniformly bounded by
\cref{eq:uniform-Linf} and
\cref{prop:weight-convergence}.
Hence
\begin{equation}\label{eq:weighted-energy-bound}
    \int_{B_{\Lambda'}}P_{R_j}|\nabla u_j|^2\dd y\leq C.
\end{equation}
Using \cref{eq:conformal-lower-bound}, we obtain
\begin{equation}\label{eq:ordinary-H1-bound}
    \|u_j\|_{H^1(B_{\Lambda'})}\leq C.
\end{equation}
The Banach--Alaoglu and Rellich compactness theorems now give
the first assertion.
\medskip

\noindent
{\bf Step 3:}
Fix a compact set
$K\subset B_\Lambda\setminus\{0\}$ and
$\varphi\in C_c^\infty(K)$.  We claim that
\begin{equation}\label{eq:flux-error-goes-zero}
    \int_K(P_{R_j}-P_\infty)
    \nabla u_j\cdot\nabla\varphi\dd y
    \to0.
\end{equation}
In fact, for $\varepsilon>0$, consider
\[
    E_j^\varepsilon
    :=\{y\in K:|P_{R_j}-P_\infty|>\varepsilon\},
    \qquad
    G_j^\varepsilon:=K\setminus E_j^\varepsilon.
\]
By \cref{eq:AR-L1-ball}, we have
\begin{equation}\label{eq:bad-set-properties}
    \epsilon|E_j^\varepsilon|
    \leq \int_{B_\Lambda }|P_{R_j}-P_{\infty}|
    \to0,
    \qquad
    \text{as }j\to+\infty.
\end{equation}
Notice that $P_\infty$ is bounded on $K$, it follows from \cref{eq:bad-set-properties} that
\begin{equation}\label{eq. small integration on bad set}
    \int_{E_j^\varepsilon}P_{R_j}\dd y
    \leq
    \int_{E_j^\varepsilon}|P_{R_j}-P_\infty|
    +\int_{E_j^\varepsilon}P_\infty
    \to0\quad
    \text{as }j\to+\infty.
\end{equation}
On the good set $G_j^\epsilon$, by \cref{eq:ordinary-H1-bound},
\[
\left|
\int_{G_j^\varepsilon}(P_{R_j}-P_\infty)
\nabla u_j\cdot\nabla\varphi
\right|
\leq C\varepsilon.
\]
On the bad set, the $P_{R_j}$-part of \cref{eq:flux-error-goes-zero} is estimated by Cauchy--Schwarz, \cref{eq:weighted-energy-bound} and \cref{eq. small integration on bad set}:
\[
\left|
\int_{E_j^\varepsilon}P_{R_j}
\nabla u_j\cdot\nabla\varphi
\right|
\leq
\left(\int_{E_j^\epsilon}P_{R_j}|\nabla u_j|^2\right)^{1/2}
\left(\int_{E_j^\varepsilon}
P_{R_j}|\nabla\varphi|^2\right)^{1/2}
\to0.
\]
Since $P_\infty$ is bounded on $K$, the remaining part of \cref{eq:flux-error-goes-zero} satisfies, by \cref{eq:bad-set-properties} and \cref{eq:ordinary-H1-bound},
\[
\left|
\int_{E_j^\varepsilon}P_\infty
\nabla u_j\cdot\nabla\varphi
\right|
\leq
C\|\nabla\varphi\|_{L^\infty(K)}
|E_j^\varepsilon|^{1/2}
\|\nabla u_j\|_{L^2(K)}
\to0.
\]
Combining the above three estimates, and then letting
$\varepsilon$ go to zero, proves the claim
\cref{eq:flux-error-goes-zero}.

Since, for any $\varphi\in C_c^\infty(K)$,
\[
    \int_{K} P_{R_j}\nabla u_j\cdot\nabla\varphi=0,
\]
we have
\[
    \int_{K} P_\infty\nabla u_j\cdot\nabla\varphi\to0,\quad 
    \text{as }j\to+\infty.
\]
Using the weak $H^1$ convergence of $\{u_j\}$ and the fact that $P_\infty$ is bounded
on $K$, we conclude
\begin{equation}\label{eq. punctured equation}
    \int_{K} P_\infty\nabla u_\infty\cdot\nabla\varphi=0.
\end{equation}

{\bf Step 4:}
Notice that the local $L^\infty$ bound \cref{eq:uniform-Linf} passes to the limit, so $u_\infty$ is bounded
near $0$. 
 Let $\chi_\varepsilon$ be a radial cutoff function satisfying
\[
    \chi_\varepsilon=0\text{ on }B_\varepsilon,
    \qquad
    \chi_\varepsilon=1\text{ on }\mathbb R^n\setminus B_{2\varepsilon},
    \qquad
    |\nabla\chi_\varepsilon|\leq C\varepsilon^{-1}.
\]
Then a direct calculation yields
\begin{equation}\label{eq:vertex-capacity}
    \int_{B_\Lambda} P_\infty|\nabla\chi_\varepsilon|^2\dd y
    \leq
    C\varepsilon^{-2}
    \int_\varepsilon^{2\varepsilon}
    r^{n-1-m(n-2)}\dd r
    \leq C\varepsilon^{(1-m)(n-2)}
    \to 0.
\end{equation}
Choose an outer cutoff function $\eta\in C_c^\infty(B_\Lambda)$.
Testing the equation \cref{eq. punctured equation} with
$\eta^2\chi_\varepsilon^2u_\infty$ and using the boundedness of
$u_\infty$ gives
\[
    \int_{B_\Lambda} P_\infty\eta^2\chi_\varepsilon^2
    |\nabla u_\infty|^2
    \leq
    C\|u_\infty\|_{L^\infty(\operatorname{supp}\eta)}^2
    \left(
    \int_{B_\Lambda} P_\infty|\nabla\eta|^2
    +\int_{B_\Lambda} P_\infty|\nabla\chi_\varepsilon|^2
    \right).
\]
Letting $\varepsilon$ go to zero, we see from \cref{eq:vertex-capacity} that
$P_\infty^{1/2}\nabla u_\infty$ is locally square integrable across the
origin. Furthermore,  for an arbitrary
$\varphi\in C_c^\infty(B_\Lambda)$, use
$\varphi\chi_\varepsilon$ as a test function in  equation \cref{eq. punctured equation}.
Then, by \cref{eq:vertex-capacity}, we have
\begin{align*}
    \left|
    \int_{B_\Lambda}P_{\infty}\chi_\epsilon
    \nabla u_{\infty}\cdot\nabla \varphi
    \right|
    =&\left|
    \int_{B_\Lambda} P_\infty\varphi
    \nabla u_\infty\cdot\nabla\chi_\varepsilon
    \right|\\
    \leq &
\|\varphi\|_{L^\infty (\text{supp}\nabla \chi_\epsilon)}
\left(
\int_{B_{2\varepsilon}}P_\infty|\nabla u_\infty|^2
\right)^{1/2}
\left(
\int_{B_{2\epsilon}} P_\infty|\nabla\chi_\varepsilon|^2
\right)^{1/2}\to0.
\end{align*}
This proves the second assertion.

{\bf Step 5:}
Fix $0<r<\Lambda$.  By \cref{eq:uniform-Linf}, there holds
\[
    |u_j|+|u_\infty|\leq M_r
    \qquad\text{on }B_r.
\]
After passing to the almost-everywhere convergent subsequence, we have
\[
    |u_j^2-u_\infty^2|\to0
    \quad\text{almost everywhere in }B_r.
\]
It follows from the dominated convergence theorem and \cref{eq:BR-L1-ball} that
\begin{equation}
\int_{B_r}|u_j^2-u_\infty^2|Q_{R_j}\dd y
\leq
\int_{B_r}|u_j^2-u_\infty^2|Q_\infty\dd y
+2M_r^2
\|Q_{R_j}-Q_\infty\|_{L^1(B_r)}
\to0.
\label{eq:weighted-square-convergence}
\end{equation}
 Combining with the local boundedness of $u_\infty$ in $B_\Lambda$, it follows that
\[
    \int_{B_r}u_j^2\dd \nu_{R_j}
    \to
    \int_{B_r}u_\infty^2\dd \nu_\infty.
\]
Also, by \cref{eq:BR-L1-ball}, there holds
\[
    \nu_{R_j}(B_r)\to\nu_\infty(B_r)>0.
\]
This proves \cref{eq:weighted-mass-convergence}.

For finitely many sequences, the products
$u_{p,j}u_{q,j}$ are locally uniformly bounded and converge almost
everywhere to $u_{p,\infty}u_{q,\infty}$.  Repeating
\cref{eq:weighted-square-convergence} proves
\cref{eq:cross-mass-convergence}.

\end{proof}

\section{Simultaneous blow-down}
\label{sec:positive-avr-dimension}

In this section, we combine the distance comparison established in
Section~3 with the weighted convergence theorem from
\cref{sec:weighted-G} to prove the sharp dimension estimate in the
the case $\beta>0$. As discussed above, each $g$-harmonic function $h$ can be properly blown-down to a limit which is a harmonic function on $\mathcal{C}_a$. To get a dimension estimate, one still needs to ensure that this process is injective from $\mathcal{H}_d(g)$ to $\mathcal{H}_d(\mathcal{C}_a)$. The standard $L^2$ inner product can be used on balls of fixed radius to maintain linear independence. Therefore it is desirable to find a common sequence of radius $R_q=2^q R_0\to \infty$ (see Corollary \ref{cor:fixed-factor-control}), along which $g$-harmonic functions are simultaneously blown-down. This technique is inspired by Lin's argument for asymptotically conic elliptic operators \cite[Theorem 2.13]{L}. Moreover, this technique also helps to establish the surjectivity of blow-down process, thus leading to a precise value of $\mathcal{H}_d(g)$. 

Let $g=e^{2w}g_0$ be a complete conformal metric on $\mathbb R^n$, $n\geq3$, satisfying $\operatorname{Ric}_g\geq0$, and suppose that its asymptotic volume ratio $\beta$ is positive.  Then the asymptotic exponent of the conformal factor is $m\in[0,1)$, and we set $a:=1-m.$

By the exact volume-ratio formula recalled in the introduction \cref{eq. volume-ratio formula},
\[
    \beta=a^{n-1}.
\]
If $m=0$, then $\beta=1$; the equality case in the Bishop--Gromov volume
comparison implies that $(\mathbb R^n,g)$ is isometric to Euclidean space.
Thus the only case in which a strict dimension drop can occur is
\begin{equation}\notag
    0<m<1,
    \qquad 0<a<1,
\end{equation}
which we assume througout the section.

For any real number $d\geq0$, recall
\[
    \mathcal H_d(g)
    :=\left\{
    h:\Delta_gh=0,
    \quad |h(x)|\leq C_h\bigl(1+d_g(0,x)^d\bigr)
    \right\}.
\]
For a nonzero $g$-harmonic function $h$, define its normalized weighted
$L^2$ mass on the Euclidean ball $B_R\subset\mathbb R^n$ by
\begin{equation}\label{eq:def-Sh}
    S_h(R)
    :=
\frac{\displaystyle\int_{B_R}h^2\dd\mu_g}
        {\displaystyle\mu_g(B_R)},
\end{equation}
where $d\mu_g:=e^{nw(x)}\dd x$ is the volume measure of $(\mathbb{R}^n,g)$. This captures the $L^2$ growth rate of $h$ on Euclidean balls.

\subsection{The homogeneous spectrum of the limiting cone}
\ \\
We start with an analysis of the harmonic functions on the limiting cone \cref{eq. def of limiting cone}.
For each integer $\ell\geq0$, let $\mathcal Y_\ell$ be the space of spherical harmonics of degree $\ell$
on $\mathbb S^{n-1}$, and let
\[
    d_\ell:=\dim\mathcal Y_\ell
    =\binom{n+\ell-1}{\ell}-\binom{n+\ell-3}{\ell-2},
    \qquad
    \lambda_\ell:=\ell(\ell+n-2).
\]
Choose an orthonormal basis
$\{Y_{\ell,q}\}_{q=1}^{d_\ell}$ of $\mathcal Y_\ell$ in
$L^2(\mathbb S^{n-1},d\omega)$, where $d\omega:=|\mathbb S^{n-1}|^{-1}d\theta$ is the normalized spherical measure. We have the following homogeneous expansion of harmonic functions on the limiting cone.

\begin{proposition}
\label{prop:cone-homogeneous-expansion}
For each integer $\ell\geq0$, let $\sigma_\ell\geq0$ be the positive solution of
\[
    \sigma_\ell\bigl(\sigma_\ell+a(n-2)\bigr)
    =\ell(\ell+n-2).
\]
Equivalently,
\begin{equation}\label{eq:sigma-explicit}
    \sigma_\ell
    =\frac{-a(n-2)+
    \sqrt{a^2(n-2)^2+4\ell(\ell+n-2)}}{2}.
\end{equation}
Every $L_\infty$-harmonic function $v$ that is locally bounded at the cone
vertex has an expansion
\begin{equation}\label{eq:cone-expansion}
    v(r,\theta)
    =\sum_{\ell=0}^{\infty}\sum_{q=1}^{d_\ell}
    c_{\ell,q}r^{\sigma_\ell}Y_{\ell,q}(\theta),
    \qquad r=|x|,
\end{equation}
with smooth convergence on compact annuli.  Moreover,
\begin{equation}\label{eq:cone-average-expansion}
    \operatorname{Av}_{\infty,r}(v^2)
    =\sum_{\ell,q}
    \frac{an}{2\sigma_\ell+an}
    |c_{\ell,q}|^2r^{2\sigma_\ell}.
\end{equation}
Consequently, the function
\[
    t\longmapsto\log \operatorname{Av}_{\infty, e^t}(v^2)
\]
is convex whenever $v\not\equiv0$.
\end{proposition}

\begin{proof}
In Euclidean polar coordinates the limiting equation
$L_\infty v=0$ is
\begin{equation}\label{eq:limit-equation-polar}
    v_{rr}
    +\frac{1+a(n-2)}{r}v_r
    +\frac1{r^2}\Delta_{\mathbb S^{n-1}}v=0.
\end{equation}
Expanding $v(r,\cdot)$ in spherical harmonics and projecting
\cref{eq:limit-equation-polar} onto $\mathcal Y_\ell$ gives the 
ODE
\[
    b''+\frac{1+a(n-2)}r b'
    -\frac{\lambda_\ell}{r^2}b=0.
\]
By setting $b(r)=r^t$, the above equation becomes a  quadratic equation of $t$ and the two roots are $\sigma_\ell$ and
$-a(n-2)-\sigma_\ell$. Then the local boundedness of $v$ at the vertex excludes the negative root. 
Standard spherical harmonic estimates give smooth
convergence on compact annuli.
This proves \cref{eq:cone-expansion}.

Since
\[
    d\nu_\infty
    =r^{-mn}r^{n-1}\dd r\dd \theta
    =r^{an-1}\dd r\dd \theta,
\]
orthogonality with respect to the normalized measure $d\omega$ yields
\cref{eq:cone-average-expansion}.  After setting $r=e^t$, the right-hand
side of \cref{eq:cone-average-expansion} is a positive sum of exponentials $e^{2\sigma_\ell t}$.  The logarithm
of such a sum is convex by a direct  calculation of the second derivative and the Cauchy-Schwarz inequality.
\end{proof}
{
\begin{lemma}
\label{lem:spectral-threshold}
Let $a\in(0,1)$. For  any positive number $d>0$, there exists a unique integer \begin{equation}\label{eq. Id upper bound}
    1\leq \bar{d} \leq
    \begin{cases}
         d &\text{ if } d\in\mathbb{Z},\\
         \lfloor d\rfloor+1
         &\text{ if } d\not\in\mathbb{Z}.
    \end{cases}
\end{equation}
such that
\begin{equation}\label{eq. def of Id}
    \sigma_{\bar{d}-1}\leq ad<\sigma_{\bar{d}},
\end{equation}
where  $\sigma_{\ell}$ is defined in \cref{eq:sigma-explicit}.
\end{lemma}
\begin{proof}
Since $\sigma_0=0$, the existence and uniqueness of $\bar{d}$ satisfying \cref{eq. def of Id} are immediate.
For every integer $d\geq1$, a direct calculation by \cref{eq:sigma-explicit} yields
\[
    \sigma_d>ad,
\]
since $0<a<1$.
It follows that $\bar{d}-1<d$, and hence $\bar{d}\leq d$.

For $d>0$ and $d\not\in\mathbb{Z}$,  Since
\[
    \sigma_{\bar{d}-1}\leq
    ad<a(\lfloor d\rfloor +1)<\sigma_{\lfloor d\rfloor +1},
\]
we conclude that $\bar{d}\leq \lfloor d\rfloor +1$.
\end{proof}
Using the notion of $\bar{d}$, we now proceed to compute the dimension of the space of polynomial-growth harmonic functions on the limiting cone.
\begin{corollary}\label{cor. dim of harmonic on cone}
    Let $a\in(0,1)$. For any positive number $d>0$,  the polynomial-growth harmonic functions on the cone \cref{eq. def of limiting cone}
satisfies
\[
    \dim\mathcal H_d(\mathcal{C}_a)
    = \dim\mathcal H_{\bar{d}-1}(\mathbb{R}^n,g_0),
\]
where $\bar{d}$ is the integer defined in \cref{lem:spectral-threshold}.
\end{corollary}
\begin{proof}
    Note that the intrinsic distance on $\mathcal{C}_a$ and the Euclidean distance is related by
    \[
        \rho(x)=\frac{1}{a}|x|^a.
    \]
    Therefore, for $u\in \mathcal H_d(\mathcal{C}_a)$, there holds
    \[
        u(x)\leq C|x|^{ad},\quad 
        \forall |x|\geq 1.
    \]
    Then \cref{eq:cone-expansion} and \cref{eq. def of Id} yield that
    \[
        \dim\mathcal H_d(\mathcal{C}_a)
        =\sum_{i=0}^{\bar{d}-1}d_i
        =\dim\mathcal H_{\bar{d}-1}(\mathbb{R}^n,g_0).
    \]
\end{proof}
}

For a $g$-harmonic function $h$ on $(\mathbb{R}^n,g)$, we shall consider the following blow-down sequence of it:
\begin{equation}\label{def of blow down sequence}
    u_R(y):=\frac{h(Ry)}{S_h(R)^{1/2}},
\end{equation}
where $S_h(R)$ is defined in \cref{eq:def-Sh}.
{The delicacy of the normalization \cref{def of blow down sequence} lies in the connection between the weighted integral average of the blow-down sequence and  the $L^2$ growth ratio of the original function $h$:
}
\begin{equation}\label{eq:exact-scaling-Sh}
    \operatorname{Av}_{R,r}(u_R^2)
    =\frac{S_h(rR)}{S_h(R)},
    \qquad \forall r>0.
\end{equation}
In fact, by the change of variables,
\[
    \operatorname{Av}_{R,r}(u_R^2)
    =\frac{\int_{B_r}h^2(Ry) Q_R(y)\dd y}{S_h(R)\int_{B_r}Q_R(y)dy}
    =\frac{\int_{B_{Rr}}h^2(x)e^{nw(x)}\dd x}
    {S_h(R)\int_{B_{Rr}}e^{nw(x)}\dd x}
    =\frac{S_h(rR)}{S_h(R)}.
\]
{
This scaling identity is the bridge between the growth rate of $h$ and the
weighted convergence theorem (\cref{thm:weighted-G-compactness}).
We shall also see that the choice of normalization \cref{def of blow down sequence} is crucial for the dimension estimates by ensuring the injectivity of the blow-down limits of such sequences.
}

\subsection{Uniform upper bound of the $L^2$ growth ratio}
\ \\
We now use \cref{thm:weighted-G-compactness} to prove a propagation of the upper bound of the $L^2$ growth ratio in \cref{eq:exact-scaling-Sh}. This is an analogue of \cite[(2.15)]{L} in our setting.
\begin{lemma}
\label{prop:one-step-doubling}
Let $N>0$ satisfy
\[
    N\notin\{\sigma_0,\sigma_1,\ldots\}.
\]
Then there exists $R_0=R_0(N,g)>0$ such that, for every $g$-harmonic function
$h$ and every $R\geq R_0$ satisfying $ S_h(2R)\leq2^{2N}S_h(R)$, there holds
\[
    S_h(R)\leq2^{2N}S_h(R/2).
\]
where $S_h$ is defined in \cref{eq:def-Sh}.
\end{lemma}

\begin{proof}
Suppose for contradiction that the assertion is false.  Then there exist $R_j\to\infty$ and
nonzero $g$-harmonic functions $h_j$ such that
\[
    S_{h_j}(2R_j)\leq2^{2N}S_{h_j}(R_j),\quad
    \text{but}\quad
    S_{h_j}(R_j/2)<2^{-2N}S_{h_j}(R_j).
\]
Define
\[
    u_j(y):=\frac{h_j(R_jy)}{S_{h_j}(R_j)^{1/2}}.
\]
The scaling identity \cref{eq:exact-scaling-Sh} gives
\[
    \operatorname{Av}_{R_j,1}(u_j^2)=1,
    \qquad
    \operatorname{Av}_{R_j,2}(u_j^2)\leq2^{2N},
    \qquad
    \operatorname{Av}_{R_j,1/2}(u_j^2)<2^{-2N}.
\]
Applying \cref{thm:weighted-G-compactness} on $B_2$,  after passing to
a subsequence, $u_j$ converges locally to a nonzero $L_\infty$-harmonic
function $u_\infty$ in $B_2$. 
Moreover, \cref{eq:weighted-mass-convergence} gives
\begin{equation}\label{eq:limit-masses-one-half-one}
    \mathrm{Av}_{\infty,1}(u_\infty^2)=1,
    \qquad
    \mathrm{Av}_{\infty,\frac{1}{2}}(u_\infty^2)
    \leq2^{-2N}.
\end{equation}
For every $r<2$, monotonicity of integration over balls gives
\[
    \operatorname{Av}_{R_j,r}(u_j^2)
    \leq
    \frac{\nu_{R_j}(B_2)}{\nu_{R_j}(B_r)}
    \operatorname{Av}_{R_j,2}(u_j^2).
\]
Using \cref{prop:weight-convergence}, \cref{thm:weighted-G-compactness} and then letting
$j\to\infty$ yields
\[
   \operatorname{Av}_{\infty,r}(u_{\infty}^2)
    \leq
    \frac{\nu_\infty(B_2)}{\nu_\infty(B_r)}2^{2N}.
\]
Letting $r$ go to $2$, we obtain
\begin{equation}\label{eq:limit-mass-two}
    \operatorname{Av}_{\infty,2}(u_{\infty}^2)
    \leq2^{2N}.
\end{equation}

Consider
\[
    F(t):=\log \operatorname{Av}_{\infty, e^t}(u_\infty^2)
\]
Then
\cref{eq:limit-masses-one-half-one} and \cref{eq:limit-mass-two} could be rewritten as
\[
    F(0)=0,
    \qquad
    F(\log 2)\leq2N\log 2,
    \qquad
    F(-\log 2)\leq-2N\log 2.
\]
By \cref{prop:cone-homogeneous-expansion}, convexity of $F$ yields
\[
    0=F(0)
    \leq\frac{F(-\log 2)+F(\log 2)}2
    \leq0.
\]
Thus equality holds throughout and $F$ is affine on
$[-\log 2,\log 2]$ with slope $2N$.  By
\cref{eq:cone-average-expansion}, the logarithm is affine only when all
nonzero homogeneous components have the same exponent. Then the slope of $F$ forces $2N=2\sigma_l$ for some $l\geq 0$, contradicting
$N\notin\{\sigma_0,\sigma_1,\ldots\}$. 
This finishes the proof.
\end{proof}

From now on, fix a positive number $d>0$. It follows from \cref{lem:spectral-threshold} that we can  choose a positive number $N$ satisfying
\begin{equation}\label{eq:choice-of-N}
   \sigma_{\bar{d}-1}\leq ad<N<\sigma_{\bar{d}}.
\end{equation}

Using distance comparison from Section~3, we turn intrinsic growth rate to a growth rate with respect to the Euclidean distance.
\begin{lemma}
\label{lem:intrinsic-to-Euclidean-growth}
Let $h\in\mathcal H_d(g)$. Then, for every $\varepsilon>0$, there exist constants
$C_{h,\varepsilon}$ and $R_{h,\varepsilon}$ such that
\begin{equation}\label{eq:Euclidean-pointwise-growth}
    \sup_{|x|\leq R}|h(x)|
    \leq C_{h,\varepsilon}R^{d(a+\varepsilon)},
    \quad \forall R\geq R_{h,\epsilon}.
\end{equation}
Consequently, there exist $N'<N$ and $C_h<\infty$, such that
\[
    S_h(R)
    \leq C_{h}R^{2N'},\quad\forall R\geq R_{h,\epsilon}.
\]
\end{lemma}

\begin{proof}
Theorem\ref{thm:distance} states that
\[
    \frac{\ln r(x)}{\ln|x|}\to a
    \qquad\text{as }|x|\to\infty.
\]
Hence, for every $\varepsilon>0$, there exists $R_\varepsilon$ such that
\[
    r(x)\leq |x|^{a+\varepsilon}
    \qquad\text{whenever }|x|\geq R_\varepsilon.
\]
After increasing the constant to absorb the compact region
$B_{R_\varepsilon}$, we obtain
\[
    \sup_{|x|\leq R} r(x)
    \leq C_\varepsilon R^{a+\varepsilon}
\]
for all sufficiently large $R$.  The defining growth bound for
$h\in\mathcal H_d(g)$ now gives \cref{eq:Euclidean-pointwise-growth}.
Finally, it follows  from \cref{eq:Euclidean-pointwise-growth} that
\[
    S_h(R)\leq C_{h,\epsilon}R^{2d(a+\varepsilon)}.
\]
Since $N>ad$, choose $\varepsilon>0$ so small that $N':=d(a+\varepsilon)<N$ and the conclusion follows.
\end{proof}

\begin{proposition}
\label{prop:uniform-dyadic-doubling}
Let $R_0=R_0(N,g)$ be the radius in
\cref{prop:one-step-doubling}.  Then, for every
$h\in\mathcal H_d(g)$ and every integer $q\geq0$,
\begin{equation}\label{eq:uniform-dyadic-doubling}
    S_h(2^{q+1}R_0)
    \leq2^{2N}S_h(2^qR_0).
\end{equation}
\end{proposition}

\begin{proof}
Fix a nonzero $h\in\mathcal H_d(g)$. We claim that   there are arbitrarily large integers $Q$
for which
\begin{equation}\label{eq:existence-good-dyadic-scales}
    S_h(2^{Q+1}R_0)
    \leq2^{2N}S_h(2^QR_0).
\end{equation}
In fact, if the reverse strict inequality held for all $Q\geq Q_0$, an iteration would yield
\[
    S_h(2^QR_0)
    \geq 2^{2N(Q-Q_0)}S_h(2^{Q_0}R_0),
\]
contradicting \cref{lem:intrinsic-to-Euclidean-growth}.

Starting from any $Q$ satisfying
\cref{eq:existence-good-dyadic-scales}, apply
\cref{prop:one-step-doubling} and we get that \cref{eq:uniform-dyadic-doubling} holds for any $0\leq q\leq Q$. Since we proved that \cref{eq:existence-good-dyadic-scales} holds for arbitrarily large integers $Q$, the proof is complete.
\end{proof}
Now we are ready to derive the following crucial $L^2$ growth ratio estimate, which would pass the growth order of a harmonic function to a uniform upper bound of the weighted integral average of the blow-down sequence by \cref{eq:exact-scaling-Sh}.

\begin{corollary}
\label{cor:fixed-factor-control}
For every fixed $T\geq1$, there exists $C_T<\infty$, depending only on
$T,N,$ and $g$, such that,
for every nonzero $h\in\mathcal H_d(g)$ and every dyadic radius
$R=2^qR_0$,
\begin{equation}\label{eq:fixed-factor-control}
    \frac{S_h(TR)}{S_h(R)}\leq C_T
\end{equation}
  Moreover, for every integer $s\geq0$,
\begin{equation}\label{eq:exact-dyadic-growth-control}
    \frac{S_h(2^sR)}{S_h(R)}\leq2^{2Ns}.
\end{equation}
\end{corollary}

\begin{proof}
The second assertion follows by iterating
\cref{eq:uniform-dyadic-doubling}.  For the first assertion, let
 $k\geq0$ be the integer satisfying $2^{k-1}<T\leq2^k$.  Since the integral of
$h^2$ is monotone under inclusion of balls,
\[
    S_h(TR)
    \leq
    \frac{\mu_g(B_{2^kR})}{\mu_g(B_{TR})}
    S_h(2^kR).
\]
By the change of variables used in \cref{eq:exact-scaling-Sh}, we have
\[
    \frac{\mu_g(B_{2^kR})}{\mu_g(B_{TR})}
    =\frac{\nu_R(B_{2^k})}{\nu_R(B_T)}.
\]
Then \cref{eq:conformal-lower-bound} and \cref{prop:weight-convergence}  show that this ratio is bounded for all $R\geq R_{2^k}$ by a constant depending
only on $T$ and $g$.  Enlarging the constant to cover the finitely many
remaining dyadic radii and using
\cref{eq:exact-dyadic-growth-control} proves
\cref{eq:fixed-factor-control}.
\end{proof}

\subsection{Simultaneous blow-down and the dimension estimate}
\ \\
We can now complete the proof of Theorem \ref{thm:T2}. We divide into two parts: Theorem~\ref{thm:injective} and Theorem~\ref{thm:surjective}.

\begin{theorem}
\label{thm:injective}
Let $g=e^{2w}g_0$ be a complete non-flat conformal metric on
$\mathbb R^n$, $n\geq3$, satisfying $\operatorname{Ric}_g\geq0$ and
$\beta\in (0,1)$.  Then, for any positive number $d>0$,
\begin{equation}\label{eq:dimension-drop-final}
    \dim\mathcal H_d(g)
    \leq \dim\mathcal H_{d}(\mathcal{C}_a)
    =\dim \mathcal{H}_{\bar{d}-1}(g_0)\leq 
    \begin{cases}
        \dim\mathcal H_{d-1}(g_0), &\text{if }
        d\in\mathbb{Z},\\
        \dim\mathcal H_d(g_0),
        &\text{if }
        d\not\in\mathbb{Z},
    \end{cases}
\end{equation}
where $a=\beta^{\frac{1}{n-1}}$.
\end{theorem}

\begin{proof}
Let $V\subset\mathcal H_d(g)$ be an arbitrary finite-dimensional subspace,
and let $k:=\dim V.$
Choose $N$ as in \cref{eq:choice-of-N}, let $R_0$ be the
radius from \cref{prop:one-step-doubling}, and set $R_j:=2^jR_0.$
For each $j$, define an inner product on $V$ by
\[
    \langle h_1,h_2\rangle_{R_j}
    :=\frac{1}{\mu_g(B_{R_j})}
      \int_{B_{R_j}}h_1h_2\dd\mu_g.
\]
It is positive definite: if $\langle h,h\rangle_{R_j}=0$, then the smooth
function $h$ vanishes on the open ball $B_{R_j}$, and unique continuation
implies $h\equiv0$.

Choose an $\langle\cdot,\cdot\rangle_{R_j}$-orthonormal basis
$ \{h_{1,j},\ldots,h_{k,j}\}$ of $V$, 
and define the blow-down sequence
\[
    u_{p,j}(y):=h_{p,j}(R_jy),
    \qquad 1\leq p\leq k.
\]
The normalization gives $S_{h_{p,j}}(R_j)=\langle h_{p,j},h_{p,j}\rangle_{R_j}=1$ and
\begin{equation}\label{eq:orthogonality-before-limit}
    \operatorname{Av}_{R_j,1}(u_{p,j}u_{q,j})
    =\langle h_{p,j},h_{q,j}\rangle_{R_j}
    =\delta_{pq}.
\end{equation}

For each fixed $T\geq1$, \cref{cor:fixed-factor-control} gives
\[
    \operatorname{Av}_{R_j,T}(u_{p,j}^2)
    =\frac{S_{h_{p,j}}(TR_j)}{S_{h_{p,j}}(R_j)}
    \leq C_T,
\]
where $C_T$ is independent of $p$ and $j$.  Therefore, we could apply
\cref{thm:weighted-G-compactness} simultaneously to the finite
family $\{u_{p,j}\}_{p=1}^k$ on $B_T$, first for each integer
$T>1$ and then by a diagonal argument.  After passing to a common
subsequence, we obtain globally defined harmonic functions
\[
    u_{1,\infty},\ldots,u_{k,\infty}
\]
on the limiting cone \cref{eq. def of limiting cone}.  Then \cref{eq:cross-mass-convergence} and \cref{eq:orthogonality-before-limit} give
\[
    \operatorname{Av}_{\infty,1}
    (u_{p,\infty}u_{q,\infty})
    =\delta_{pq}.
\]
Hence the limiting functions are linearly independent.

The sharper dyadic estimate
\cref{eq:exact-dyadic-growth-control} gives, for every integer $s\geq0$,
\[
    \operatorname{Av}_{R_j,2^s}(u_{p,j}^2)
    \leq2^{2Ns}.
\]
Passing to the limit in a ball of radius larger than $2^s$ yields
\begin{equation}\label{eq:dyadic-bound-after-limit}
    \operatorname{Av}_{\infty,2^s}(u_{p,\infty}^2)
    \leq2^{2Ns}.
\end{equation}
Expand $u_{p,\infty}$ as in
\cref{prop:cone-homogeneous-expansion}.  
By \cref{eq:cone-average-expansion} and \cref{eq:dyadic-bound-after-limit}, 
 every $u_{p,\infty}$ belongs to
\[
    \bigoplus_{\sigma_\ell<N}
    \left\{r^{\sigma_\ell}Y:Y\in\mathcal Y_\ell\right\}.
\]
By \cref{lem:spectral-threshold} and the choice of $N$ (see \cref{eq:choice-of-N}), the dimension of this space is
\[
    \sum_{\ell=0}^{\bar{d}-1}d_\ell.
\]
Since the $k$ functions $u_{1,\infty},\ldots , u_{k,\infty}$ are linearly independent and the subspace $V\subset \mathcal H_d(g) $ is arbitrary, we get:
\[
    \dim\mathcal H_d(g)
    \leq\sum_{\ell=0}^{\bar{d}-1}d_\ell
    =\dim\mathcal H_{\bar{d}-1}(g_0)
    \leq 
    \begin{cases}
        \dim\mathcal H_{d-1}(g_0), &\text{if }
        d\in\mathbb{Z},\\
        \dim\mathcal H_{\lfloor d\rfloor}(g_0)
        =\dim\mathcal H_{d}(g_0),
        &\text{if }
        d\not\in\mathbb{Z},
    \end{cases}
\]
where we used \cref{eq. Id upper bound} in the last inequality. Combining this with \cref{cor. dim of harmonic on cone} proves \cref{eq:dimension-drop-final}.
\end{proof}

\subsection{Construction of polynomial growth harmonic functions}
\ \\
In this subsection, we shall show that the reverse of the first inequality in \cref{eq:dimension-drop-final} also holds true. We start with a global convergence result for solutions to the equations $\operatorname{div}(P_R\nabla u)=0$, which upgrades the interior convergence in \cref{thm:weighted-G-compactness} by identifying the limiting boundary trace.

\begin{lemma}\label{lem. convergence up to bdry}
    Let $R_j\to\infty$. Assume  $\Lambda\geq1$ and $k\in\mathbb{Z}_+$. Let $u_j$ be weak solutions of
\begin{equation}\label{eq:uj-equation boundary value}
\begin{cases}
    \operatorname{div}(P_{R_j}\nabla u_j)=0
    &\text{in }B_\Lambda,\\
    u_j=\sum_{i=1}^k c_{i,j}Y_i
    &\text{on }\partial B_\Lambda.
\end{cases}
\end{equation}
where $Y_i,\ 1\leq i\leq k,$ are fixed smooth functions on $\partial B_\Lambda$ and $c_{i,j}$ are constants with a uniform bound
\[
    |c_{i,j}|\leq C,\qquad 
    \forall i\in\{1,\ldots,k\},\  
    \forall j\geq 1.
\]
Assume
\[
    \sup_j\operatorname{Av}_{R_j,\Lambda}(u_j^2)<\infty.
\]
Then, after passing to a subsequence, there exists a function $u_\infty$
satisfying
\[
    u_j\rightharpoonup u_\infty
    \quad\text{in }H^1(B_{\Lambda}),\quad 
    \text{and }\quad 
    u_{\infty}
    =\sum_{i=1}^k c_{i,\infty}Y_i
    \quad 
    \text{on }\partial B_\Lambda
\]
in the Sobolev trace sense, where
\[
    c_{i,\infty}:=\lim_{j\to+\infty} c_{i,j},\quad \text{for } i=1,\dots,k.
\]
\end{lemma}
\begin{proof}
Since $c_{i,j}$ are uniformly bounded, a diagonal argument yields that, up to a subsequence, 
\[
    c_{i,j}\to c_{i,\infty}\quad 
    \text{for every }i=1,\ldots,k.
\]
Then it is immediate  to see that
\begin{equation}\label{eq. L2 convergence on boundary}
    u_j\to \sum_{i=1}^k c_{i,\infty}Y_i
    \quad
    \text{in } L^{2}(\partial B_{\Lambda}).
\end{equation}

With a slight abuse of notation, we smoothly extend $Y_i$ into $B_\Lambda$ and still denote them by $Y_i,\ i=0,\dots,k$.
Consider
\[
     X_j:=\sum_{i=1}^k c_{i,j}Y_i .
\]
It follows from the uniform boundedness of $c_{i,j}$ that
\begin{equation}\label{ineq. C1 bound of boundary value}
    \|X_j\|_{C^1(\overline{B_\Lambda})}
    \leq C.
\end{equation}

We test the weak equation \cref{eq:uj-equation boundary value} with $w_j:=u_j-X_j\in H^1_0(B_\Lambda)$ and  get
\begin{align*}
    \int_{B_\Lambda} P_{R_j}
    |\nabla w_j|^2\dd x
    =&-\int_{B_\Lambda} P_{R_j}
    \nabla X_j\cdot \nabla w_j\dd x \\
    \leq &
    \left(
        \int_{B_\Lambda} P_{R_j}
        |\nabla X_j|^2\dd x
    \right)^{1/2}
    \left(
        \int_{B_\Lambda} P_{R_j}
        |\nabla w_j|^2\dd x
    \right)^{1/2}.
\end{align*}
Therefore,
\[
    \int_{B_\Lambda} P_{R_j}
    |\nabla w_j|^2\dd x
    \le
    \int_{B_\Lambda} P_{R_j}
    |\nabla X_j|^2\dd x .
\]
Combining with \cref{ineq. C1 bound of boundary value}, \cref{eq:conformal-lower-bound} and \cref{prop:weight-convergence}, it follows that
\[
    \int_{B_\Lambda} |\nabla w_j|^2\dd x
    \leq C,
\]
and hence $u_j=w_j+X_j$ are bounded in $H^1(B_\Lambda)$.
Therefore, up to a subsequence, there exists $u_\infty\in H^1(B_\Lambda)$ such that
\[
    u_j\rightharpoonup u_\infty,\quad
    \text{in } H^1(B_\Lambda).
\]
Consider the trace operator
\[
    T:H^1(B_\Lambda)\to L^2(\partial B_\Lambda)
\]
The weak convergence in $H^1(B_\Lambda)$
implies weak convergence of the traces:
\[
    Tu_j\rightharpoonup Tu_\infty
    \qquad
    \text{in }L^2(\partial B_\Lambda).
\]
Combining with \cref{eq. L2 convergence on boundary}, we get
\[
    Tu_\infty
    =\sum_{i=1}^k c_{i,\infty}Y_i
\]
as desired.
\end{proof}
Now we are ready to determine the dimension of the space of polynomial-growth harmonic functions.
\begin{theorem}\label{thm:surjective}
    Let $g=e^{2w}g_0$ be a complete non-flat conformal metric on
$\mathbb R^n$, $n\geq3$, satisfying $\operatorname{Ric}_g\geq0$ and
$\beta\in (0,1)$.  Then, for any positive number $d>0$,
\[
    \dim\mathcal H_d(\mathbb{R}^n,g)
    \geq\dim\mathcal H_d(\mathcal{C}_a)
\]
where $a=\beta^{\frac{1}{n-1}}$ and $\mathcal{C}_a$ is the limiting cone defined in \cref{eq. def of limiting cone}.
\end{theorem}
\begin{proof}
    Let $\bar{d}$ be the integer defined in \cref{lem:spectral-threshold}.
    For brevity, denote $\dim\mathcal H_{\bar{d}-1}(g_0)$ by $D$. We will inductively construct linearly independent harmonic functions $h_1,\dots, h_D\in \mathcal H_d(\mathbb{R}^n,g)$, so that
    $\dim\mathcal H_d(\mathbb{R}^n,g)
    \geq D$.

    {\bf Step 1:} First, we let $h_1\equiv 1\in \mathcal H_d(\mathbb{R}^n,g)$. Suppose that we have constructed linearly independent harmonic functions $h_1,\dots, h_k$, where $k\leq D-1$. For clarity, we relabel the first $D$ eigen-exponents $\sigma_\ell$ \cref{eq:sigma-explicit} with multiplicity and the corresponding spherical harmonics $Y_{\ell,i}$ as
    \[
        0=\mu_1<\mu_2\leq \mu_3\leq \cdots\leq \mu_D=\sigma_{\bar{d}-1},\quad 
        1\equiv Y_1,  Y_2,Y_3,\dots, Y_D.
    \]
    The spherical harmonics $Y_i$ are ordered such that $|x|^{\mu_i}Y_i$ is a harmonic function on the limiting cone $\mathcal{C}_a$.

Fix a positive number $N$ such that 
\begin{equation}\label{eq. choice of N in construction}
    \mu_D=\sigma_{\bar{d}-1}<N<\sigma_{\bar{d}},
\end{equation}
and let $R_0=R_0(N,g)$ be the constant in \cref{prop:one-step-doubling}.
Define $R_j:=2^jR_0$.
Let $v_j\in C^{\infty}(B_{R_j}(0))$ solve the $g$-harmonic equation
\[
    \begin{cases}
        \operatorname{div}
        (e^{(n-2)w}\nabla v_j)=0
        &\text{in }B_{R_j}(0),\\
        v_j=\sum_{i=1}^{k+1} a_{i,j} R_j^{\mu_i}Y_i
        &\text{on }\partial B_{R_j}(0),
    \end{cases}
\]
where $a_{1,j},\dots, a_{k+1,j}$ are constants such that
\begin{equation}\label{eq. choice of a ij}
    \frac{1}
    {\operatorname{Vol}_g(B_{R_0}(0))}
    \int_{B_{R_0}(0)}v_j^2\dd vol_g=1,\quad 
    \frac{1}
    {\operatorname{Vol}_g(B_{R_0}(0))}
    \int_{B_{R_0}(0)}v_jh_i\dd vol_g=0,\quad 
    \forall i=1,2,\dots, k.
\end{equation}
Consider the blow-down sequence 
\[
    u_j(y):=\frac{v_j(R_jy)}{S_{v_j}(R_j)^{\frac{1}{2}}},\quad 
    j\geq 1.
\]
It satisfies $\operatorname{Av}_{R_j,1}(u_j^2)=1$ and  solves
\begin{equation}\label{eq. boundary value eq for uj}
    \begin{cases}
        \operatorname{div}
        (P_{R_j}\nabla u_j)=0
        &\text{in }B_{1}(0),\\
        u_j=\sum_{i=1}^{k+1}
        \frac{a_{i,j} R_j^{\mu_i}}{S_{v_j}(R_j)^{\frac{1}{2}}}Y_i
        &\text{on }\partial B_{1}(0).
    \end{cases}
\end{equation}

{\bf Step 2:} 
We claim that the constants 
\[
    c_{i,j}:=\frac{a_{i,j} R_j^{\mu_i}}{S_{v_j}(R_j)^{\frac{1}{2}}}
\]
are uniformly bounded as $j\to+\infty$. Otherwise, 
\[
    M_j:=\frac{\sqrt{\sum_{i=1}^{k+1} a_{i,j}^2R_j^{2\mu_i}}}
    {S_{v_j}(R_j)^{\frac{1}{2}}}\to +\infty,
    \quad\text{as }j\to+\infty.
\]
Then we have
\[
    \begin{cases}
        \operatorname{div}
        \left(
        P_{R_j}\nabla \frac{u_j}{M_j}
        \right)=0
        &\text{in }B_{1}(0),\\
        \frac{u_j}{M_j}=\sum_{i=1}^{k+1}
        \frac{a_{i,j} R_j^{\mu_i}}
        {\sqrt{\sum_{i=1}^{k+1} a_{i,j}^2R_j^{2\mu_i}}}Y_i
        &\text{on }\partial B_{1}(0),
    \end{cases}
\]
with $\operatorname{Av}_{R_j,1}\left(
(\frac{u_j}{M_j})^2\right)=\frac{1}{M_j^2}\to 0.$ 
On the one hand, by \cref{thm:weighted-G-compactness} and the same argument as in the proof of \cref{prop:one-step-doubling}, there exists a function $u_\infty$ on $B_1(0)$ such that $\frac{u_j}{M_j}\rightharpoonup \tilde u_\infty$ in $H^1(B_r(0))$ for any $r<1$, and $\operatorname{Av}_{\infty,1}(\tilde u_\infty^2)
    =0.$ Therefore,
\[
    \tilde u_\infty\equiv 0\quad \text{a.e. in }
    B_1(0).
\]
 On the other hand, by \cref{lem. convergence up to bdry}, $\frac{u_j}{M_j}\rightharpoonup \tilde u_\infty$ in $H^1(B_1(0))$ and 
\[
    \tilde u_\infty=\sum_{i=1}^{k+1} c_i Y_i\quad 
    \text{on }\partial B_1(0)
\]
for some constants $c_1,\dots, c_{k+1}$ with $c_1^2+\cdots+c_{k+1}^2=1$. By the Sobolev trace inequality, this is impossible and the proof of the claim finishes.

{\bf Step 3:} Applying \cref{thm:weighted-G-compactness} and \cref{lem. convergence up to bdry} to \cref{eq. boundary value eq for uj}, we know that there exists $u_\infty\in H^1(B_1(0))$ such that
\[
    L_{\infty}u_\infty=0\quad 
    \text{in } B_1(0),
\]
and 
\[
    u_{\infty}
    =\sum_{i=1}^{k+1} c_{i,\infty}Y_i
    \quad 
    \text{on }\partial B_1(0)
\]
for some constants $c_{1,\infty},\dots, c_{k,\infty}$. Therefore, by \cref{prop:cone-homogeneous-expansion}, 
\[
    u_\infty(y)
    =\sum_{i=1}^{k+1}
    c_{i,\infty}|y|^{\mu_i}
    Y_i\left(\frac{y}{|y|}\right) \quad 
    \text{in }B_1(0).
\]
Then a direct calculation yields
\begin{align*}
    \operatorname{Av}_{\infty,1/2}(u_\infty^2)
    =&\frac{\int_{B_{1/2}(0)}u_\infty^2(y)
    |y|^{-(1-a)n}\dd y}
    {\int_{B_{1/2}(0)}|y|^{-(1-a)n}\dd y}
    =\sum_{i=1}^{k+1}\frac{an}{an+2\mu_i}
    |c_{i,\infty}|^2
    (\frac{1}{2})^{2\mu_i}\notag\\
    \geq& \sum_{i=1}^{k+1}\frac{an}{an+2\mu_i}
    |c_{i,\infty}|^2
    (\frac{1}{2})^{2\mu_D}
    =2^{-2\mu_D}
    \operatorname{Av}_{\infty,1}(u_\infty^2).
\end{align*}
Noticing $\mu_D=\sigma_{\bar{d}-1}$ and the choice of $N$ \cref{eq. choice of N in construction}, the above inequality implies
\[
    \operatorname{Av}_{\infty,1}(u_\infty^2)
    \leq 2^{2N} 
    \operatorname{Av}_{\infty,1/2}(u_\infty^2).
\]
Therefore, by \cref{thm:weighted-G-compactness}, for all sufficiently large $j$, there holds
\[
    1=\operatorname{Av}_{R_j,1}(u_j^2)
    \leq 2^{2N} 
    \operatorname{Av}_{R_j,1/2}(u_j^2).
\]
By \cref{eq:exact-scaling-Sh}, this could be rewritten as
\[
    S_{v_j}(R_j)
    \leq 2^{2N}S_{v_j}
    \left(\frac{1}{2}R_j\right).
\]
Therefore, by \cref{prop:one-step-doubling} and the normalization \cref{eq. choice of a ij}, we have
\[
    S_{v_j}(2^sR_0)
    \leq 2^{2Ns}S_{v_j}(R_0)=2^{2Ns},\quad 
    \forall s=0,\dots, j.
\]
Consequently, $\{v_j\}$ are a sequence of $g$-harmonic functions on the exhaustion $B_{R_j}(0)$, with a uniform upper bound of $L^2$ norm:
\begin{equation}\label{ineq. uniform upper bound L2 norm}
    \frac{1}
    {\operatorname{Vol}_g(B_{2^sR_0}(0))}
    \int_{B_{2^sR_0}(0)}v_j^2\dd vol_g
    \leq 2^{2Ns},\quad 
    \forall j\geq s.
\end{equation}
{\bf Step 4:} 
By the mean-value inequality \cite[Theorem 2.1]{LS} and the gradient estimate for harmonic functions (see \cite[Theorem 1.1]{L2}), $v_j$ converges uniformly on compact sets of $(\mathbb{R}^n,g)$ to a $g$-harmonic function $v_\infty$ (see also \cite[Lemma 2.1]{L2}). 

The normalization \cref{eq. choice of a ij} yields that $v_\infty$ is linearly independent of $h_1,\dots, h_k$.
Moreover, $v_\infty\in \mathcal H_{\frac{N}{a}}(\mathbb{R}^n,g)$ by virtue of \cref{ineq. uniform upper bound L2 norm} and the distance comparison \cref{thm:distance}. Going through the proof of \cref{thm:injective} shows that $v_\infty\in \mathcal H_{d}(\mathbb{R}^n,g)$. Therefore, we could choose $h_{k+1}$ as $v_\infty$ and the proof finishes.
\end{proof}

\section{Completion of the proof of Theorem\ref{thm:A}}
\label{sec:completion-lcf}

We now pass from the globally conformal case treated in the preceding
sections to an arbitrary complete locally conformally flat manifold with nonnegative Ricci curvature.  The remaining cases are elementary. We recall that the classification theorems of Zhu \cite[Theorem 1]{Z} and of
Carron--Herzlich \cite[Theorem A]{CH} imply that a connected complete locally conformally flat
manifold $(M^n,g)$, $n\geq3$, with $\operatorname{Ric}_g\geq0$ belongs to
one of the following classes:
\begin{enumerate}
\item $M$ is non-flat and globally conformally equivalent to $\mathbb R^n$;
\item $M$ is globally conformally equivalent to a spherical space form;
\item $M$ is locally isometric to the round cylinder
      $\mathbb R\times\mathbb S^{n-1}$;
\item $M$ is a complete flat manifold.
\end{enumerate}

We shall discuss case by case and thereby complete the proof of Theorem~\ref{thm:A}. For reader's convenience and simplicity, we recall it here with a statement only for integer growth rate $d$. We refer the reader to the Remark ~\ref{rem:noninteger} below for non-integer growth rate. 
\begin{theorem}
\label{thm:main-lcf-comparison}
Let $(M^n,g)$, $n\geq3$, be a connected complete locally conformally flat
manifold satisfying $\operatorname{Ric}_g\geq0$.  Then, for every integer
$d\geq0$,
\begin{equation}\label{eq:main-lcf-comparison}
    h_d(M)\leq h_d(\mathbb R^n).
\end{equation}
Moreover, if equality holds for some integer $d\geq1$, then $(M,g)$ is
isometric to the Euclidean space. 
\end{theorem}

\begin{remark} \label{rem:noninteger}
By \cref{thm:injective}, it is easy to see that  \cref{eq:main-lcf-comparison} also holds for non-integer $d$.

However, one cannot expect the rigidity. In fact, by \cref{thm:surjective}, 
 for every $d\in(1,2)$, there exists $a\in(0,1)$ such that $\bar{d}=2$, so that the corresponding conformal metric $g$ on $\mathbb{R}^n$ is non-flat and satisfies 
\[
\operatorname{Ric}_g\geq 0,
\qquad
\beta_g>0,
\]
and
\[
\dim \mathcal H_d(\mathbb R^n,g)
=\dim \mathcal{H}_{d}(\mathcal{C}_a)=
\dim \mathcal H_{\bar{d}-1}(\mathbb R^n,g_0)
=
n+1.
\]
\end{remark}

\subsection{The globally conformal non-flat case}
\ \\
Suppose first that $M$ belongs to class~(1). Thus
\[
    (M^n,g)\cong (\mathbb R^n,e^{2w}g_0)
\]
with $g$ non-flat.  Let $\beta$ denote its asymptotic volume ratio.

If $\beta=0$, Theorem \ref{thm: T1} gives
\[
    \mathcal H_d(M)=\mathbb R
\]
for every finite $d$. Hence
\[    h_d(M)=1\leq h_d(\mathbb R^n),\quad 
    \forall d\geq 0,
\]
and the inequality is strict whenever $d\geq1$.

If $\beta\in(0,1)$, Theorem \ref{thm:T2} gives, for every
integer $d\geq1$,
\[
    h_d(M)
    \leq h_{d-1}(\mathbb R^n)
    < h_d(\mathbb R^n).
\]
For $d=0$, both spaces consist of constants. Thus the theorem holds in
this class, and equality for a positive degree is impossible.

\subsection{The spherical case}
\ \\
Suppose that $M$ belongs to class~(2). A spherical space form is compact,
and global conformal equivalence preserves compactness.  Hence $M$ is
compact.  Every harmonic function on a connected compact manifold is
constant, so
\[
    \mathcal H_d(M)=\mathbb R
    \qquad\text{for every }d\geq0.
\]
Consequently,
\[
    h_d(M)=1\leq h_d(\mathbb R^n),
\]
with strict inequality for every $d\geq1$.

For the remaining two cases, the universal cover of $M$ is either isometric to the standard cylinder $C=\mathbb{R}\times \mathbb{S}^{n-1}$ or the standard $\mathbb{R}^n$, where one has clear understanding of $\mathcal{H}_d(\widetilde{M})$. The lift of a harmonic function on $M$ to $\widetilde{M}$ preserves the polynomial growth rate and is deck transformation group invariant. Using these facts, one also has  an explicit description for $\mathcal{H}_{d}(M)$.

\subsection{The cylindrical case}
\ \\
We first classify polynomial-growth harmonic functions on the standard cylinder $\mathbb{R}\times \mathbb{S}^{n-1}$.
\begin{lemma}
\label{lem:cylinder-harmonic-functions}
Let $C$ be the standard cylinder with the product metric $dt^2+g_{\mathbb S^{n-1}}$,
then every harmonic function on $C$ having polynomial growth is of the form
\begin{equation}\label{eq:cylinder-affine-form}
    u(t,\theta)=A+Bt.
\end{equation}
In particular,
\[
    \dim\mathcal H_d(C)
    =
    \begin{cases}
        1,&0\leq d<1,\\
        2,&d\geq1.
    \end{cases}
\]
\end{lemma}

\begin{proof}
Let $\{Y_{\ell,\alpha}\}$ be an orthonormal basis of spherical harmonics satisfying
\[
    -\Delta_{\mathbb S^{n-1}}Y_{\ell,\alpha}
    =\lambda_\ell Y_{\ell,\alpha},
    \qquad
    \lambda_\ell=\ell(\ell+n-2).
\]
For a harmonic function $u$ on $C$, set
\[
    a_{\ell,\alpha}(t)
    :=\int_{\mathbb S^{n-1}}
      u(t,\theta)Y_{\ell,\alpha}(\theta)\,d\theta.
\]
Since
\[
    \Delta_C=\partial_t^2+\Delta_{\mathbb S^{n-1}},
\]
each coefficient satisfies
\begin{equation}\label{eq:cylinder-mode-ode}
    a_{\ell,\alpha}''-\lambda_\ell a_{\ell,\alpha}=0.
\end{equation}
The polynomial-growth assumption on $u$ implies the same polynomial bound
for every coefficient $a_{\ell,\alpha}$ as $t\to\pm\infty$.

For $\ell=0$, equation \cref{eq:cylinder-mode-ode} gives
\[
    a_{0}''=0,
\]
so $a_0(t)=A+Bt$.  If $\ell\geq1$, then
\[
    a_{\ell,\alpha}(t)
    =c_{\ell,\alpha}^{+}e^{\sqrt{\lambda_\ell}t}
     +c_{\ell,\alpha}^{-}e^{-\sqrt{\lambda_\ell}t}.
\]
Polynomial growth as $t\to+\infty$ forces
$c_{\ell,\alpha}^{+}=0$, while polynomial growth as $t\to-\infty$ forces
$c_{\ell,\alpha}^{-}=0$.  Thus every positive spherical mode vanishes, and
\cref{eq:cylinder-affine-form} follows.  The final dimension assertion is
immediate.
\end{proof}

We then exploit the structure of the deck transformation group to determine
$\mathcal H_d(M)$ precisely. If $M$ is compact, the computation is trivial, so we mainly focus in the case that $M$ is complete noncompact. 

Let $M$ be a complete noncompact
manifold locally isometric to $C$. Since $n\geq3$, $C$ is simply connected and is
therefore the Riemannian universal cover of $M$.  Thus
\[
M=C/\Gamma,
\]
where $\Gamma\subset \operatorname{Isom}(C)$ is a discrete group acting freely and
properly discontinuously on~$C$.

Note the product decomposition of $ C$ is preserved by every
isometry.  Hence
\[
\operatorname{Isom}( C)
=
\operatorname{Isom}(\mathbb R)\times O(n),
\]
and every $\gamma\in\Gamma$ has the form
\[
\gamma(t,\theta)
=
\bigl(\varepsilon_\gamma t+a_\gamma,
A_\gamma\theta\bigr),
\]
where $\varepsilon_\gamma\in\{\pm1\},
a_\gamma\in\mathbb R,\text{ and }
A_\gamma\in O(n)$.

Consider the natural projection
\[
p:\Gamma\longrightarrow\operatorname{Isom}(\mathbb R),
\qquad
p(\gamma)(t)=\varepsilon_\gamma t+a_\gamma.
\]

\begin{proposition}\label{prop:cylindrical-case}
Let $M^n$, $n\geq3$, be a complete noncompact manifold locally
isometric to the round cylinder
$\mathbb R\times\mathbb S^{n-1}$.  Then the image
$p(\Gamma)$ is either trivial or a group of order two generated by a
reflection of $\mathbb R$.  Moreover,
\[
h_d(M)
=
\begin{cases}
1, & 0\leq d<1,\\[1mm]
2, & d\geq1 \ \text{and } p(\Gamma)=\{e\},\\[1mm]
1, & d\geq1 \ \text{and } p(\Gamma)\neq\{e\}.
\end{cases}
\]
\end{proposition}

\begin{proof}
We first describe the possible image of the deck transformation group
on the $\mathbb R$-factor.  The subgroup $p(\Gamma)$ is discrete in
$\operatorname{Isom}(\mathbb R)$.  Indeed, if there were a sequence of
distinct elements $\gamma_j\in\Gamma$ such that $p(\gamma_j)$ converged
in $\operatorname{Isom}(\mathbb R)$, then, since $O(n)$ is compact, a
subsequence of $A_{\gamma_j}$ would also converge.  It would follow
that $\gamma_j\gamma_k^{-1}$ can be made arbitrarily close to the
identity for large $j,k$, contradicting the discreteness of $\Gamma$.

Similarly, $F:=\ker p$
is finite, since it is a discrete subgroup of the compact group $O(n)$.
Note that the discrete subgroups of $\operatorname{Isom}(\mathbb R)$ are, up to
conjugation,
\[
\{e\},\qquad
\mathbb Z_2,\qquad
\mathbb Z,\qquad
D_\infty,
\]
where $\mathbb Z_2$ is generated by a reflection and $D_\infty$ is the
infinite dihedral group. Since  the last two groups act cocompactly on
$\mathbb R$, and $\mathbb S^{n-1}$ is compact, if $p(\Gamma)$
contains a nontrivial translation, then
$ C/\Gamma$ would be compact.  Then the noncompactness of $M$ implies
\[
p(\Gamma)=\{e\}
\qquad\text{or}\qquad
p(\Gamma)\cong\mathbb Z_2.
\]

If $p(\Gamma)=\{e\}$, then every deck transformation acts only on the
sphere:
\[
\gamma(t,\theta)=(t,A_\gamma\theta).
\]
Thus
\[
M\cong
\mathbb R\times\bigl(\mathbb S^{n-1}/F\bigr),
\]
where $F\subset O(n)$ is a finite group acting freely on
$\mathbb S^{n-1}$.

If $p(\Gamma)\cong\mathbb Z_2$, then $\Gamma$ contains an element
$\gamma_0$ whose action on the $\mathbb R$-factor is a reflection, 
  which has the form
\[
\gamma_0(t,\theta)=(-t+a,A\theta)
\]
for some $A\in O(n)$.

We now turn to polynomial growth harmonic functions.  Let
$u\in\mathcal H_d(M)$ and let
\[
\widetilde u:=u\circ\pi
\]
be its lift to $C$.  For fixed base points
$\widetilde p\in C$ and $p=\pi(\widetilde p)$, notice
\[
d_M\bigl(p,\pi(\widetilde x)\bigr)
\leq
d_{ C}(\widetilde p,\widetilde x),
\]
so $\widetilde u$ has polynomial growth of degree at most $d$ on
$ C$.  By \cref{lem:cylinder-harmonic-functions}, we have
\begin{equation}\label{eq. harmonic on universal cover}
    \widetilde u(t,\theta)=A+Bt.
\end{equation}
Since $\widetilde u$ is the lift of $u$ via the universal covering, it must satisfy
\[
\widetilde u\circ\gamma=\widetilde u
\qquad\text{for every }\gamma\in\Gamma.
\]

If $p(\Gamma)=\{e\}$, the function $t$ itself is $\Gamma$-invariant
and hence descends to $M$.  Consequently,
\[
\mathcal H_d(M)
=
\begin{cases}
\mathbb R, & 0\leq d<1,\\[1mm]
\operatorname{span}\{1,t\}, & d\geq1,
\end{cases}
\]
where, by abuse of notation, $t$ denotes the descended function on
$M$.

If $p(\Gamma)\cong\mathbb Z_2$, consider a reflection element of $\Gamma$
\[
\gamma_0(t,\theta)=(-t+a,A\theta).
\]
The invariance of \cref{eq. harmonic on universal cover} under $\gamma_0$ gives
 $B=0.$
Hence every polynomial-growth harmonic function on $M$ is constant. Then the formula for $h_d(M)$ is immediate. 
\end{proof}

\subsection{The flat case}
\ \\
Finally we consider the class~(4).  Let $M^n$ be a connected complete flat manifold. then its universal cover is $\mathbb{R}^n$. By the Cheeger-Gromoll soul theorem, $M$ admits a compact
totally geodesic flat soul $\Sigma^k$.  Since the inclusion
$\Sigma \hookrightarrow M$ is a homotopy equivalence, it induces an isomorphism on fundamental groups. Then the inverse image of $\Sigma$ in the universal cover $\mathbb R^n$ is connected and is
therefore an affine subspace $V\cong\mathbb R^k$.  The group of deck transformation 
$\Gamma=\pi_1(M)$ preserves $V$ and acts cocompactly on it. By  virtue of the Bieberbach theorem,
after identifying
\[
\mathbb R^n=V\oplus V^\perp,
\]
each $\gamma\in\Gamma$ takes the form
\[
\gamma(x,y)
=
\bigl(A_\gamma x+a_\gamma,B_\gamma y\bigr),
\qquad
A_\gamma\in O(k),\quad a_{\gamma} \in \mathbb{R}^k, \quad 
B_\gamma\in O(n-k).
\]
Consequently, $M$ is the total space of a flat Riemannian vector
bundle over the compact flat manifold $\Sigma^k=V/\Gamma$.

Let
\[
K:=\overline{\{B_\gamma:\gamma\in\Gamma\}}
\subset O(n-k)
\]
be the closure of the normal holonomy group.  For $\ell\geq0$, denote
by $\mathcal Y_\ell(\mathbb R^{n-k})$ the space of homogeneous harmonic
polynomials of degree $\ell$ on $\mathbb R^{n-k}$.

\begin{proposition}
Let $M^n$ be a connected complete flat manifold and let the notation be
as above.  For every real number $d\geq0$, putting $D=\lfloor d\rfloor$, there is a
natural isomorphism
\[
\mathcal H_d(M)
\cong
\bigoplus_{\ell=0}^{D}
\mathcal Y_\ell(\mathbb R^{n-k})^K,
\]
where the superscript $K$ denotes the $K$-invariant subspace.
Consequently,
\[
h_d(M)
=
\sum_{\ell=0}^{D}
\dim\mathcal Y_\ell(\mathbb R^{n-k})^K
\leq
h_d(\mathbb R^{n-k})
\leq
h_d(\mathbb R^n).
\]
In particular, if $h_d(M)=h_d(\mathbb R^n)$
for some real number $d\geq1$, then $M^n$ is isometric to $\mathbb R^n$.
\end{proposition}

\begin{proof}
Let
$
\pi:\mathbb R^n\longrightarrow M
$
be the universal covering map and let $u\in\mathcal H_d(M)$. Then its lift
$\widetilde u:=u\circ\pi$
is harmonic on $\mathbb R^n$. 
As discussed in Proposition ~\ref{prop:cylindrical-case} above, the lift $\widetilde u$ has polynomial growth of degree at most $d$, and
\[
\widetilde u\circ\gamma=\widetilde u
\qquad\text{for every }\gamma\in\Gamma.
\]

We claim that $\widetilde u$ is independent of the $V$-variable.
Indeed, since $\Sigma=V/\Gamma$ is a compact flat manifold, Bieberbach's
theorem provides a finite-index subgroup $\Lambda\subset\Gamma$ whose
induced action on $V$ consists of translations by a full lattice in
$V$.  Thus, for $\lambda\in\Lambda$,
\[
\lambda(x,y)
=
(x+a_\lambda,B_\lambda y),
\]
where $B_\lambda\in O(n-k)$ and the vectors $a_\lambda$ span $V$.

Fix such a $\lambda\in \Lambda$.  The $\Gamma$-invariance of $\widetilde u$ gives
\[
\widetilde u(x+ja_\lambda,B_\lambda^j y)
=
\widetilde u(x,y),
\qquad \forall j\in\mathbb Z.
\]
Since $B_\lambda\in O(n-k)$, there exists a sequence
$j_i\to\infty$ such that $B_\lambda^{j_i}\longrightarrow I.$
Suppose for contradiction that $\widetilde u$ had positive polynomial degree in the
$a_\lambda$-direction, then, after dividing the preceding identity by
the highest corresponding power of $j_i$ and letting
$i\to\infty$, its leading coefficient in that direction would have
to vanish.  This implies 
\[
\partial_{a_\lambda}\widetilde u=0,\quad  \forall a_{\lambda}\in V,
\]
from which we conclude that
\[
\widetilde u(x,y)=P(y)
\]
for some harmonic polynomial $P$ on $V^\perp\cong \mathbb R^{n-k}$.
Furthermore, the $\Gamma$-invariance of $\tilde{u}$ becomes
\[
P(B_\gamma y)=P(y),
\qquad \gamma\in\Gamma.
\]
By continuity this is equivalent to
\[
P(By)=P(y),
\qquad \forall B\in K.
\]
Thus $P$ is a $K$-invariant harmonic polynomial on $\mathbb R^{n-k}$.

Conversely, every $K$-invariant harmonic polynomial $P$ on
$\mathbb R^{n-k}$, regarded as a function on
$V\oplus V^\perp$ by
\[
(x,y)\longmapsto P(y),
\]
is $\Gamma$-invariant and hence descends to a harmonic function on
$M$.  Its polynomial degree agrees with its growth degree on $M$.
Indeed, if $\Sigma\subset M$ denotes the soul, then
\[
d_M(\pi(x,y),\Sigma)=|y|,
\]
while, for a fixed $p\in \Sigma$,
\[
|y|
\leq
d_M(\pi(x,y),p)
\leq
|y|+\operatorname{diam}(\Sigma).
\]
It follows that
\[
\mathcal H_d(M)
\cong
\bigoplus_{\ell=0}^{D}
\mathcal Y_\ell(\mathbb R^{n-k})^K.
\]

Taking dimensions gives
\[
h_d(M)
=
\sum_{\ell=0}^{D}
\dim\mathcal Y_\ell(\mathbb R^{n-k})^K
\leq
\sum_{\ell=0}^{D}
\dim\mathcal Y_\ell(\mathbb R^{n-k})
=
h_d(\mathbb R^{n-k})
\leq h_d(\mathbb{R}^n).
\]

Finally, suppose that the equality $h_d(M)=h_d(\mathbb R^n)$
holds for some real number $d\geq1$. Then  the above inequalities imply $n-k=n$, $K=\{e\}$ and
$M \cong \mathbb{R}^n$
isometrically.
\end{proof}


\end{document}